\documentclass[12pt]{amsart}
\usepackage{latexsym, amsmath, amssymb,amsthm,amsopn,amsfonts}
\usepackage{version}
\usepackage{epsfig,graphics,color,graphicx,graphpap,dsfont}
\usepackage{amssymb}

\ifx\pdfoutput\undefined
  \DeclareGraphicsExtensions{.eps}
\else
  \ifx\pdfoutput\relax
    \DeclareGraphicsExtensions{.eps}
  \else
    \ifnum\pdfoutput>0
      \DeclareGraphicsExtensions{.pdf}
    \else
      \DeclareGraphicsExtensions{.eps}
    \fi
  \fi
\fi

\newcommand{\LV}{\left|}
\newcommand{\RV}{\right|}
\newcommand{\LN}{\left\|}
\newcommand{\RN}{\right\|}
\newcommand{\LB}{\left[}
\newcommand{\RB}{\right]}

\newcommand{\LC}{\left(}
\newcommand{\RC}{\right)}
\newcommand{\LA}{\left<}
\newcommand{\RA}{\right>}

\newcommand{\be}{\begin{equation}}
\newcommand{\ee}{\end{equation}}
\newcommand{\bes}{\begin{equation*}}
\newcommand{\ees}{\end{equation*}}

\newcommand{\al}{\alpha}
\newcommand{\N}{\mathcal{N}}

\theoremstyle{plain}

\newtheorem{thm}{Theorem}
\newtheorem{prop}{Proposition}[section]

\newtheorem{lem}[prop]{Lemma}

\theoremstyle{definition}

\newtheorem{rem}{Remark}[section]

\numberwithin{equation}{section}

\def\squarebox#1{\hbox to #1{\hfill\vbox to #1{\vfill}}}

\newcommand{\p}{\partial}
\newcommand{\R}{\mathbb{R}}
\usepackage{amsxtra}

\title
[The flow map for the gravity-capillary equations]
{On the regularity of the flow map for the gravity-capillary equations}  

\author[R. M. Chen]
{Robin Ming Chen}
\email{mingchen@pitt.edu}

\author[J.L. Marzuola]
{Jeremy L. Marzuola}
\email{marzuola@math.unc.edu}

\author[D. Spirn]
{Daniel Spirn}
\email{spirn@math.umn.edu}

\author[J. D. Wright]
{J. Doug Wright}
\email{jdoug@math.drexel.edu}

\address{Department of Mathematics, University of Pittsburgh \\
 301 Thackeray Hall, Pittsburgh, PA  15260, USA  }

\address{Department of Mathematics, University of North Carolina \\
    CB \#3250, Phillips Hall, Chapel Hill, NC  27599, USA}

\address{School of Mathematics , University of Minnesota \\
206 Church St. S.E., Minneapolis, MN 55416, USA}

\address{Mathematics Department, Drexel University \\
33rd and Market Street, Philadelphia, PA  19104, USA}

\thanks{Acknowledgements:   R.M.C. was supported in part by NSF grant DMS--0908663.  J.L.M. was supported in part by an NSF Postdoctoral Fellowship at Columbia University and a Hausdorff Center Postdoctoral Fellowship at Universit\"at Bonn.  In addition, J.L.M. thanks the Courant Institute for graciously hosting him as a Visiting Academic during part of this research.  D.S. was supported in part by NSF grants DMS--0707714 and DMS--0955687.  J.D.W. was supported in part by NSF grants DMS--0807738, DMS--0908299 and DMS--1105635. The authors also thank Dave Ambrose, David Lannes and Fernando Reitich for helpful conversations along the way to completing this work.  We also thank the anonymous referee for a careful reading of the draft and for assisting us in putting the results in a better context.  We also thank Thomas Alazard, Nicolas Burq and Claude Zuilly for a careful reading of the draft and catching an error in an early version of the draft.
}

\begin{document}    

\begin{abstract}
We prove via explicitly constructed initial data that solutions to the gravity-capillary wave system in $\mathbb{R}^3$ representing a $2d$ air-water interface immediately fail to be $C^3$ with respect to the initial data if the initial data $(h_0, \psi_0) \in H^{s+\frac12} \otimes H^{s}$ for $s<3$.  Similar results hold in $\mathbb{R}^2$ domains with a $1d$ interface.  Furthermore, we discuss the related threshold for the pure gravity water wave system.
\end{abstract}
   
\maketitle

\section{Introduction}

The motion of a perfect, incompressible and irrotational fluid under the
influence of gravity and surface tension is described by the free surface Euler (or water-waves)
equations. In this paper we consider the fluid domain to be given by $\Omega_t = \{ (x, y, z)\in \R^3: \ z\leq h(t, x, y) \}$, where we have assumed that the the free surface is described by the graph $z = h(t, x, y)$. The incompressibility and irrotationality of the flow indicate the existence of a velocity potential $\phi$, and hence the Euler equation can be reduced to an equation on the free surface (\cite{CraigSchanzSulem1, CraigSchanzSulem2, Zakharov}). More precisely, denoting the trace of the velocity potential $\phi$ on the free surface by $\psi(t, x, y) = \phi(t, x, y, h(t, x, y))$, the equations for $\psi$ and $h$ are
\begin{equation} \label{ww}
\left\{
\begin{split}
\p_t h & = G(h) \psi \\
\p_t \psi & = - g h + \tau \operatorname{div} \LC { \nabla h \over \sqrt{ 1 + \LV \nabla h \RV^2 }} \RC 
- {1\over 2 } \LV \nabla \psi \RV^2 + { \LC G(h) \psi + \nabla h \cdot \nabla \psi \RC^2 \over 2 \LC 1 + \LV \nabla h \RV^2 \RC } , \\
\end{split}\right.
\end{equation}
with
\begin{eqnarray*}
\begin{pmatrix}
h \\
\psi 
\end{pmatrix} (0,x,y)  = \begin{pmatrix}
h_0 \\
\psi_0 
\end{pmatrix} ( x,y) \in \left\{ \begin{array}{ll} H^{s+\frac12} (\mathbb{R}^2) \otimes H^{s} (\mathbb{R}^2)& \hbox{ when } \tau > 0 \\
H^{s-\frac12} (\mathbb{R}^2) \otimes H^{s} (\mathbb{R}^2)& \hbox{ when } \tau = 0 \end{array} \right. , 
\end{eqnarray*}
where the (rescaled) Dirichlet-Neumann operator is given by
\begin{equation*}
G(h)\psi = \sqrt{1+|\nabla h|^2} \p_{\textbf{n}_+} \phi \big|_{z = h(t, x, y)},
\end{equation*} 
and  $\textbf{n}_+$ is the outward unit normal vector to the free surface.

There has been an extensive study of  the well-posedness of the water wave problem. Early works go back to Nalimov \cite{Nalimov}, Yosihara \cite{Yosihara} and Craig \cite{Craig} where $1d$ free surface are considered, surface tension is neglected and the motion of the free surface is a small perturbations of still water. The general local wellposedness of $2d$ full water wave problem was established by Wu \cite{Wu1}, see also Ambrose-Masmoudi \cite{AmbroseMasmoudi1}. 
The large time well-posedness is established in \cite{Wu3}.

In the case of $3d$ gravity water waves, Wu \cite{Wu2} proved it is local wellposedness. Lannes \cite{Lannes} considered the same problem for finite depth. 
The global wellposedness was proved by Germain-Masmoudi-Shatah \cite{GMS} and Wu \cite{Wu4} independently. 

For gravity-capillary waves, Ambrose-Masmoudi \cite{AmbroseMasmoudi2} considered the zero surface tension limit and provided a proof of wellposedness in the $3d$ case. Alazard-Burq-Zuily \cite{ABZ} studied the regularity of the local solutions using a paradifferential formulation in the $2d$ and $3d$ cases.  See also the work of Christianson-Hur-Staffilani \cite{CHS}, where along with the work of Alazard-Burq-Zuily \cite{ABZ}, smoothing and Strichartz estimates were established for the $1d$ interface problem.

When the fluid is rotational and without surface tension, the problem has been shown to be wellposed by Iguchi-Tanaka-Tani \cite{ITT} in $2d$ and Lindblad \cite{Lindblad} and Zhang-Zhang \cite{ZZ} in $3d$. Including surface tension, the wellposedness is proved by Ogawa-Tani \cite{OT1, OT2} in $2d$ and Coutand-Shkoller \cite{CS} and Shatah-Zeng \cite{SZ} in $3d$.

The water wave system \eqref{ww} generates a substantial and rich family of equations that describe solutions in various asymptotic regimes, including KdV, Camassa-Holm, BBM, among others.  These model equations  can accept remarkably low-regularity initial data, and so it is natural to ask whether the original system can also accept low-regularity data.  In this paper, rather than continuing to investigate the wellposedness of water waves, we look at on the onset of breakdown of regularity of the solution operator, or flow map.  It is shown that equations \eqref{ww} 
start failing to behave in a "reasonable" way in rather regular spaces, counter to the behavior of many of the approximate equations listed above.  
 In particular we identify the value of the regularity parameter $s$ for which the model fails to be solvable by standard iterative methods. Our result joins up with the wellposedness results of  Alazard-Burq-Zuily \cite{ABZ}, see  Theorem \ref{illposednesstheorem}.

Unlike the semilinear dispersive system, where in many cases a contraction principle can be applied even in a very low-regularity regime to obtain wellposedness, and consequently the smoothness of the solution map, the water wave problem is by nature quasilinear, suggesting that the initial profile can determine the existence time and the regularity of the solution map rather strongly. On the other hand, we will still use a perturbative approach by considering the nonlinear part of the system as a perturbation of the linear flow and show that the corresponding solution operator experiences a severe singularity in certain Sobolev spaces near the zero solution based upon those nonlinear interactions.

We prove, via explicitly constructed initial data, that the solution map for the gravity-capillary wave system in $\mathbb{R}^3$ (representing a $2d$ air-water interface) immediately fails to be $C^3$ with respect to the initial data if the data $(h_0, \psi_0)$ is in $H^{s+\frac12} \otimes H^{s}$ for $s<3$. This is  primarily due to the influence of surface tension.  
Similar results hold for a $1d$ air-water interface in $\mathbb{R}^2$.  The work, while not addressing fully the issue of ill-posedness of the gravity-capillary problem, follows from prior results related to the study of ill-posedness by Bona-Tzvetkov \cite{BonaTzvetkov}.  The idea stems from looking at the second and third order iterates in a Picard iteration about the chosen initial data, which leads to a {\em high-to-high frequency interaction}: a solution starts off initially with small energy and Fourier transform supported primarily at high frequencies, but quickly generates a large energy at high frequencies.  In contrast to the high-to-low frequency cascade as in \cite{BejenaruTao, ChristCollianderTao1} for the study of ill-posedness in Schr\"odinger equations, where the Sobolev index is negative, our high-high frequency interaction is chosen because of the ill-posedness into Sobolev spaces of positive indices. The initial data will be described in Section \ref{S:ill}, but it is chosen to maximize to the greatest extent possible the frequency interactions in the nonlinear terms.  A $C^3$ ill-posedness result was also obtained in the study of the nonlinear Schr\"odinger equation on the sphere in \cite{BGT1,BGT2}, in which the dispersive estimates are known to have a loss of derivatives due to the lack of dispersion.

We note here that since we are looking at an immediate failure of the solution map to be $C^3$ with respect to the initial data, the effects here are related very much to nonlinear interactions through the Dirichlet-Neumann map and surface tension. 
The reason we go up to $C^3$ for a breakdown of regularity is due to the regularity restriction of the Dirichlet-Neumann map -- $h$ needs to be at least Lipschitz. But the second iteration only provides breakdown of $C^2$ regularity at the threshold at $H^{2-}(\R^2)$.

It is natural to ask if this this breakdown of regularity of the solution map is in fact a failure of uniform continuity of the solution map (a stronger measure of ill-posedness), which has been studied in semilinear dispersive problems, see for instance Bourgain-Pavlovi\'c \cite{BoPa}, Christ-Colliander-Tao \cite{ChristCollianderTao} and Bejenaru-Tao \cite{BejenaruTao}.  However, 
 because of the quasilinear nature, it is very challenging to fit our problem in the framework of \cite{BejenaruTao}, where they provide a general method for $C^0$ ill-posedness of semilinear problems.   Implementation of this framework in \eqref{ww} (and hence proving that our obstruction is indeed $C^0$-ill-posed) may be achievable if the behavior of the Taylor expansion of the nonlinear terms in \eqref{ww} (which we call $\mathcal{N}$, see equation \eqref{eqn:ww2d} below) can be very well understood at all orders.

Proofs of local well-posedness for generic quasilinear equations typically involve paradifferential iteration schemes with dependence on derivatives of the solution in the inhomogeneity of each iteration as well as the iterated linearization of the evolution map.  Hence one should not expect the solution map to be smooth in relation to a particular Sobolev space $H^s$, but rather in terms of scales of those spaces such that the inhomogeneous terms can be smoothly evolved.  Indeed, given a solution $u$ to a quasilinear equation, linearizing about such a solution may involve terms of the form $\partial_x u$, which should live in a less regular Sobolev space.  In addition, the water wave problem is implicitly strongly hyperbolic in nature, meaning that one should not necessarily expect sufficient gain of regularity in the evolution through dispersive methods to make a smooth flow map possible.  Hence, for a full ill-posedness result, one would need to control the flow map at all relevant scales for the equation, which we do not do here.  The result here distinctly uses the implicitly linear structure of the underlying evolution operator and only takes such quasilinear effects into consideration in the inhomogeneous, nonlinear iteration terms.

We note that a key part of our analysis assumes the existence of a natural polynomial expansion of the Dirichlet-Neumann map $G(h)$.  This is by no means a trivial assumption and requires a fair bit of regularity to hold  to a certain order as made clear by the increasing number of derivatives appearing in the expansion of $G(h)$, see for instance \cite{CraigGroves}.  Using the results of Nicholls-Reitich \cite{NR1,NR2} on the analyticity properties of the Dirichlet-Neumann map, one can be assured of a valid expansion  up to order $k$ in $2d$ if we require $h \in H^s$ for $s \geq 2$ with
\begin{eqnarray*}
\| G_k (h) \psi \|_{H^s} \lesssim \| \psi \|_{H^{s-\frac12}}.
\end{eqnarray*}  
Such requirements are  strong, but our minimal regularity threshold moves well beyond this particular regularity requirement.  

Finally, we remark that when we consider the  pure gravity water wave system ($\tau = 0$), the construction 
indicates the solution map is not 
$C^2$ for data $(h_0, \psi_0) \in H^{s-\frac12} \otimes H^{s}$ for $s<\frac52$.  However, in order to make sense of the Dirichlet-Neumann map, one needs to consider, at the minimum, an initial height $h_0$ in the class of Lipschitz functions, and the $h_0 \in H^s$ for $s < {2}$ threshold  fails to be Lipschitz. Iterating the pure gravity system further to third order does not  move past the formal $5\over2$ critical Sobolev exponent threshold.  Thus, this result can only be considered a formal failure of the iteration process.

\subsection{Formal argument for onset of regularity breakdown} 

Though the problem is fundamentally quasilinear, one can simply ask if the breakdown of flow-map regularity is dominated by the second-order terms in the nonlinear operator $\mathcal{N}$.   Expanding  the pure surface tension  to second order 
results in a problem that can be looked at via scaling.  In particular 
 we have the model
\begin{equation} \label{modelsurfacetension} \begin{split}
h_t & = |D| \psi - \operatorname{div}( h \nabla \psi) - |D| ( h |D|\psi) \\
\psi_t & = \Delta h + {1\over 2} \LC |D| \psi \RC^2 - {1\over2} \LV \nabla \psi \RV^2 
\end{split} \end{equation}
where we expanded the Dirichlet-to-Neumann operator to second order.  If we look for 
 scale invariant solution $(h, \psi) \mapsto (h_\lambda, \psi_\lambda)$ then 
\begin{align*}
h_\lambda & = \lambda^{-{2\over3}} h( \lambda t , \lambda^{2\over3} x) \\
\psi_\lambda & = \lambda^{-{1\over3}} \psi( \lambda t , \lambda^{2\over3} x).  
\end{align*}
A short calculation shows blowup occurs \emph{formally} for ${\dot H}^{s + {1 \over 2}} \otimes {\dot H}^{s}$ with $s < \frac32$.
As we will see, this argument is insufficient to understand the obstacle to local well-posedness in either the gravity-capillary or pure surface tension  problem, since  the mean-curvature operator does not show up until third order, and at that point the associated pseudodifferential operator becomes  dominant.

\begin{rem}
We note that a similar expansion of the  pure gravity water wave equations leads to the  model system
\begin{equation*} \label{modelgravity} \begin{split}
h_t & = |D| \psi - \operatorname{div}( h \nabla \psi) - |D| ( h |D|\psi) \\
\psi_t & = - h + {1\over 2} \LC |D| \psi \RC^2 - {1\over2} \LV \nabla \psi \RV^2
\end{split} \end{equation*}
with scale invariant  solutions $(h, \psi) \mapsto (h_\lambda, \psi_\lambda)$ satisfying
\begin{align*}
h_\lambda & = \lambda^{-2} h( \lambda t , \lambda^2 x) \\
\psi_\lambda & = \lambda^{-3} \psi( \lambda t , \lambda^2 x),
\end{align*}
which blow up for data in ${\dot H}^{s - {1 \over 2}} \otimes {\dot H}^{s}$ for
$s < \frac52$.  Hence
one could conjecture $\frac52$ to be the lower regularity threshold of \eqref{ww}, see Section~\ref{sec:puregravity}.  
In this case the model system  does accurately capture the essential dominant problem to low regularity solutions.  We note, however, that $h_0 \in H^s$ for $s < 2$ fails to be Lipschitz.   
\end{rem}

\subsection{Abstract $C^k$ regularity breakdown of flow map argument}
We follow the $C^2$ iterative method of Bona-Tzvetkov \cite{BonaTzvetkov}, used to study the BBM equation.    
Extracting the linear operator in \eqref{ww} we write it as
\begin{equation*}
\left\{
\begin{split}
\p_t h & = \LV D \RV \psi + \LB  G(h) \psi - \LV D \RV \psi \RB  \\
\p_t \psi & = \LC \tau \Delta - g \RC  h + \LB \tau \LC  \operatorname{div} \LC { \nabla h \over \sqrt{ 1 + \LV \nabla h \RV^2 }} \RC  
- \Delta h \RC
\right. \\
& \quad \quad \quad  \quad \quad \quad \left. - {1\over 2 } \LV \nabla \psi \RV^2 + { \LC G(h) \psi + \nabla h \cdot \nabla \psi \RC^2 \over 2 \LC 1 + \LV \nabla h \RV^2 \RC } \RB .
\end{split}\right.
\end{equation*}
Setting 
\[
\mathcal{L} = \begin{pmatrix} 0 & |D| \\ \tau \Delta - g & 0 \end{pmatrix},
\]
we write our  nonlinear system of equations as 
\begin{eqnarray}
\label{eqn:ww2d}
\left\{ \begin{array}{c}
\p_t \begin{pmatrix}
h \\
\psi
\end{pmatrix} =  \mathcal{L} \begin{pmatrix}
h \\
\psi
\end{pmatrix} + \mathcal{N} (h,\psi), \\
 \\
 \begin{pmatrix}
h(x,0)  \\
\psi(x,0) 
\end{pmatrix} = \begin{pmatrix}
\tilde{h}_0 \\
\tilde{\psi}_0
\end{pmatrix}
\end{array} \right.
\end{eqnarray}
where 
\begin{eqnarray*}
\mathcal{N} = \begin{pmatrix}
\mathcal{N}_1 (h,\psi)  \\
\mathcal{N}_2 (h,\psi)
\end{pmatrix}  .
\end{eqnarray*}
Then the Duhamel formula implies
\begin{eqnarray}
\begin{pmatrix}
h \\
\psi
\end{pmatrix} =  e^{\mathcal{L}t} \begin{pmatrix}
\tilde{h}_0 \\
\tilde{\psi}_0
\end{pmatrix} + \int_0^t e^{\mathcal{L}(t-t')} \mathcal{N} (t') dt' .
\end{eqnarray}

The linear flow operator is determined by 
\begin{align*}
e^{\mathcal{L}t} \begin{pmatrix}
H \\
\Psi
\end{pmatrix}  =  \begin{pmatrix} {L}_1 H  +  {L}_3  \Psi \\ {L}_2  H + {L}_1  \Psi \end{pmatrix} ,
\end{align*}
where 
\begin{align*}
\widehat{L}_1(\xi, t)  &  = \cos(\lambda (|\xi|) t)  \\
\widehat{L}_2(\xi, t) & = \sin(\lambda( |\xi|)t) {|\xi| \over \lambda(|\xi|)}  \\
\widehat{L}_3(\xi, t)& =  - \sin(\lambda (|\xi| ) t){\lambda(|\xi|) \over |\xi| },  
\end{align*}
$\xi\in\R^2$ and $\lambda(r) \in \{ \lambda_{gc}, \lambda_{st}, \lambda_g\}$, where
\begin{align}
\lambda_{gc}(r) & = \sqrt{g r + \tau r^3}, \\
\lambda_{st}(r) & =  \tau^{1\over2}  r^{3\over2}, \\
\lambda_{g}(r) & = g^{1\over2}  r^{1\over2} . \label{lambdag}
\end{align}  
In the following we will choose $g, \tau \in \{0,1\}$ and hence three possible wave functions.  
Our results differ distinctly between the cases $\tau = 0$ and $\tau \neq 0$.


Following the ideas of \cite{BonaTzvetkov} and \cite{BejenaruTao} we take
$$
\begin{pmatrix}
h(x,0)  \\
\psi(x,0) 
\end{pmatrix} = \alpha \begin{pmatrix}
h_0 \\
\psi_0
\end{pmatrix}
$$
where $\alpha \ge 0$ and $h_0$ and $\psi_0$ are 
functions specified below. Of course, when $\alpha = 0$
the solution of \eqref{ww} is the trivial solution. Our
goal is to show that the map 
$(\alpha,h_0,\psi_0) \longrightarrow (h(t),\psi(t))$
is not $C^3$ with respect to $\alpha$ (when $\alpha = 0$). This in turn
implies the solution map to to \eqref{ww} is
not $C^3$ (see below).

In Duhamel form, we write \eqref{eqn:ww2d} as 
\begin{align*}
\begin{pmatrix} h \\ \psi \end{pmatrix} (\alpha,t)
= e^{ t \mathcal{L}} \alpha \begin{pmatrix}
{h}_0 \\
{\psi}_0
\end{pmatrix}
+ \int_0^t e^{(t- t') \mathcal{L}} \mathcal{N}(\alpha, t') dt'.
\end{align*}
Here have abused notation
and written
$ \mathcal{N}(\alpha, t') $ in place
of $\mathcal{N}(h(\alpha, t'),\psi(\alpha,t'))$ in order to make explicit
the dependence on $\alpha$.

From assumed wellposedness we have 
\begin{equation} \label{initial}
\begin{pmatrix}
h \\
\psi
\end{pmatrix} (0,t) = \begin{pmatrix}
0 \\ 0
\end{pmatrix} .
\end{equation}
We can continue computing higher variations.  At first order we have
\begin{equation} \label{initialderiv}
\left. \p_\alpha \begin{pmatrix}
h \\
\psi
\end{pmatrix} (\alpha,t) \right|_{\alpha=0}= e^{t \mathcal{L}} \begin{pmatrix}
h_0 \\ \psi_0
\end{pmatrix},
\end{equation}
which follows from Duhamel's formula 
\[
\p_\alpha \begin{pmatrix}
h \\
\psi
\end{pmatrix} (\alpha,t) 
= e^{t \mathcal{L}} \begin{pmatrix}
h_0 \\ \psi_0
\end{pmatrix} + \int_0^t e^{(t- t') \mathcal{L}} \p_\alpha \mathcal{N}(0, t') dt'
\]
and the expansion in Section \ref{s:Nexp} below.
For the second iterate we have 
\begin{equation}  \label{seconditerate}
\left. \p^2_\alpha \begin{pmatrix}
h \\
\psi
\end{pmatrix} (\alpha,t) \right|_{\alpha=0}= \int_0^t e^{(t - t') \mathcal{L}} \p_\alpha^2 \mathcal{N}(0,t') dt'  .
\end{equation}
Finally, for the third iterate we have 
\begin{equation}  \label{thirditerate}
\left. \p^3_\alpha \begin{pmatrix}
h \\
\psi
\end{pmatrix} (\alpha,t) \right|_{\alpha=0}= \int_0^t e^{(t - t') \mathcal{L}} \p_\alpha^3 \mathcal{N}(0,t') dt'  .
\end{equation}

If our solution map is indeed $C^k$ with respect to the initial data and the topology $X^s(\R^2)$, then necessarily
\begin{equation}
\LN \left. \p_\alpha^k \begin{pmatrix} h \\ \psi \end{pmatrix} \right|_{\alpha=0} \RN_{X^s} \leq C \LN \begin{pmatrix} h_0 \\ \psi_0 \end{pmatrix} \RN^k_{X^s} 
\end{equation} 
for a uniform constant $C$.  Therefore, to show our mapping fails to be $C^k$, we need to show that for any $C$, there exists  explicit initial data $h_0, \psi_0$ 
such that 
\begin{equation}\label{violation}
\LN \left. \p_\alpha^k \begin{pmatrix} h \\ \psi \end{pmatrix} \right|_{\alpha=0} \RN_{X^s} > C \LN \begin{pmatrix} h_0 \\ \psi_0 \end{pmatrix} \RN^k_{X^s} .
\end{equation}

The following result identifies when the solution map fails $C^3$ continuity at the origin for gravity-capillary, pure surface tension, and pure gravity water wave equations with a one or two dimensional interface.

\begin{thm}\label{illposednesstheorem}
Let $X^s \equiv H^{s+{1\over2}} \otimes H^s$ and 
assume $\tau > 0$ and $g \in [0, \infty)$. For any $T>0$ the flow map associated to \eqref{ww} is not of class $C^3$  
\begin{itemize}
\item from $X^s $ to $C([0, T]; X^s)$ for $s < 3$ when  $d = 2$.  
\item from $X^s $ to $C([0, T]; X^s)$ for $s < \frac5 2$ when 
 $d = 1$.  
\end{itemize}
\end{thm}

\begin{rem}
 Alazard-Burq-Zuily \cite{ABZ} proved that \eqref{ww} with $\tau > 0$ is well-posed in $X^s$ for all $s > 3$.  Therefore, our result provides in some sense an endpoint to local well-posedness.  
\end{rem}

\begin{rem}
\emph{Formally}, we find from our construction that the flow map associated to \eqref{ww} is not $C^3$ (or $C^2$) from  $Y^s \equiv H^{s-{1\over2}} \otimes H^s$ to $C([0, T]; Y^s)$ for $s < {5\over2}$  when $\tau =0$ and $g > 0$ and $d =2$.  
\end{rem}

In order to prove this result we will expand the operator to third order. This produces cubic terms in the Duhamel term involving the linear propagators.  Restricting our nontrivial data to $\psi_0$ with support on a high-frequency sector produces a class of data that fails $C^3$ continuity.  

\section{Expansion of $\mathcal{N}$ and $k$-iterates of \eqref{ww}}
\label{s:Nexp}

In this section, we assume the free suface $h\in H^s(\R^2)$ with $s\geq2$ so that the Dirichlet-Neumann map $G(h)$ has an analytic expansion (see Nicholls-Reitich \cite{NR1,NR2}).  The following proposition details the  quadratic and cubic approximation
of \eqref{ww} in $2d$.  

\begin{prop}\label{duhamelprop}
We have
\begin{align}
\left. \p_\alpha \mathcal{N}(\alpha,t) \right|_{\alpha=0} & = \begin{pmatrix} 0\\ 0 \end{pmatrix}, \label{1stvarN}\\
\left. \p^2_\alpha \mathcal{N} (\alpha,t)\right|_{\alpha=0} & = 
\begin{pmatrix} 
 -2 \LB \operatorname{div} \LC h_1 \nabla \psi_1 \RC + |D| \LC h_1 |D| \psi_1 \RC \RB
 \\\\  \LC |D| \psi_1 \RC ^2 - |\nabla \psi_1|^2
 \end{pmatrix}, \label{2ndvarN}
\end{align} 
where
\begin{equation}\label{h1psi1}
\begin{pmatrix} h_1 \\ \psi_1 \end{pmatrix}  =  e^{t \mathcal{L}} \begin{pmatrix} h_0 \\ \psi_0 \end{pmatrix} 
\end{equation}
and so from \eqref{seconditerate}
\begin{equation}\label{h2psi2}
\begin{pmatrix} h_2 \\ \psi_2 \end{pmatrix}  =  \int_0^t e^{(t - t') \mathcal{L}} \p_\alpha^2 \mathcal{N}(0,t') dt'   .
\end{equation}

Finally, we can define the third iterate $h_3, \psi_3$ in terms of the first and second iterate.  In particular we have
\begin{equation}
\begin{split}
\left. \p^3_\alpha \mathcal{N} (\alpha,t)\right|_{\alpha=0} & =
\begin{pmatrix} 
  \left. \p^3_\alpha \mathcal{N}_1 (\alpha,t)\right|_{\alpha=0} 
 \\\\  \left. \p^3_\alpha \mathcal{N}_2 (\alpha,t)\right|_{\alpha=0}
 \end{pmatrix}, 
 \end{split}\label{3rdvarN}
\end{equation}
where
\begin{align}
\label{cubic1}
\left. \p^3_\alpha \mathcal{N}_1 (\alpha,t)\right|_{\alpha=0} & =
-3 \left[  \operatorname{div} \left( h_1  \nabla \psi_2 \right) + |D| \left( {h_1 } |D| \psi_2 \right) \right] \notag \\
 &-3 \left[  \operatorname{div} \left( h_2  \nabla \psi_1 \right) + |D| \left( {h_2 } |D| \psi_1 \right) \right]  \\
& +3 \left[    \Delta \left( {h^{2}_1 } |D| \psi_1 \right) 
+
 |D| \left( h^2_1 \Delta \psi_1 \right)  + 2|D| \left( h_1  |D| \left( {h_1 } |D| \psi_1 \right) \right)\right] \notag
\end{align}
and 
\begin{align}
\label{cubic2}
\left. \p^3_\alpha \mathcal{N}_2 (\alpha,t)\right|_{\alpha=0} & =
-3{\tau } \operatorname{div}\left( \nabla h_1 | \nabla h_1|^2 \right) \notag \\
&- 6 |D| \psi_1 \left( h_1 \Delta \psi_1 + |D|(h_1|D|\psi_1) \right)\\
 &+3\left( [|D|\psi_1][|D|\psi_2]- \nabla \psi_1 \cdot \nabla \psi_2\right). \notag
\end{align}

\end{prop}

We prove this in the following argument.
The Dirichlet-to-Neumann map is denoted 
$$G(h) \psi = G_0 \psi + G_1(h) \psi + G_2 (h)\psi + \dots.$$
From \cite{NR2}, we have 
$$
G(0) = G_0 = | D |,
$$
where $D = - i \nabla_x$ and further from (\cite{NR2}, equation (21) on page 114) we have the recursion relation:
$$
G_n(h) \psi =  |D|^{n-1} D\left( {h^n \over n!} D \psi \right) - \sum_{l = 0}^{n-1} |D|^{n-l} \left( {h^{n-l} \over (n-l)!} G_l(h) \psi \right).
$$

Thus,
\begin{equation}\begin{split}
G_1(h) \psi = & D\left( h D \psi \right) - \sum_{l = 0}^{0} |D|^{1-l} \left( {h^{1-l} \over (1-l)!} G_l(h) \psi \right)\\
=& D\left( h D \psi \right) - |D| \left( {h } G_0(h) \psi \right)\\
=& D\left( h D \psi \right) - |D| \left( {h } |D| \psi \right).
\end{split}
\end{equation}
Using the definition of $D$:
\begin{equation}
G_1(h) \psi=-\operatorname{div} \left( h \nabla \psi \right) - |D| \left( {h } |D| \psi \right).
\end{equation}

The recursion formula with $n =2$ gives:
\begin{equation}\begin{split}
G_2(h) \psi =&  |D| D\left( {h^2 \over 2} D \psi \right) - \sum_{l = 0}^{1} |D|^{2-l} \left( {h^{2-l} \over (2-l)!} G_l(h) \psi \right)\\
=&  |D| D\left( {h^2 \over 2} D \psi \right) -\left( 
 |D|^{2} \left( {h^{2} \over 2} G_0(h) \psi \right) 
+ |D| \left( {h} G_1(h) \psi \right) 
\right).\\
\end{split}
\end{equation}
Using the formulae for $G_1$ and $G_0$ then rearranging terms:
\begin{equation}\begin{split}
=&
{1 \over 2}  |D| D\left( {h^2} D \psi \right) -\left( 
{1 \over 2}  |D|^{2} \left( {h^{2} } |D| \psi \right) 
+ |D| \left\{ {h}\left[ -\operatorname{div} \left( h \nabla \psi \right) - |D| \left( {h } |D| \psi \right)\right]\right\} 
\right)\\
=& -{1 \over 2}  |D|^{2} \left( {h^{2} } |D| \psi \right) + {1 \over 2}  |D| D\left( {h^2} D \psi \right)
+  |D| \left\{ {h}\left[ \operatorname{div} \left( h \nabla \psi \right) + |D| \left( {h } |D| \psi \right)\right]\right\} \\
=&  -{1 \over 2}  |D|^{2} \left( {h^{2} } |D| \psi \right) 
+|D| \left\{
{1 \over 2}D(h^2 D \psi) +   {h}\operatorname{div} \left( h \nabla \psi \right) + h  |D| \left( {h } |D| \psi \right)
\right\}\\
\end{split}
\end{equation}
If we use $D = -i \nabla$ in the first term in the curly braces:
\begin{equation}\label{doom}
\begin{split}
G_2(h) \psi=&  -{1 \over 2}  |D|^{2} \left( {h^{2} } |D| \psi \right) 
+|D| \left\{
-{1 \over 2}\operatorname{div} (h^2 \nabla \psi) +   {h}\operatorname{div} \left( h \nabla \psi \right) + h  |D| \left( {h } |D| \psi \right)
\right\}
\end{split}
\end{equation}

Then the product rule gives:
\begin{eqnarray*}
-{1 \over 2}\operatorname{div} (h^2 \nabla \psi) +   {h}\operatorname{div} \left( h \nabla \psi \right) & = & 
-{1 \over 2}\left( 2 h \nabla h \cdot \nabla \psi + h^2 \Delta \psi\right) + h^2 \Delta \psi + h \nabla h \cdot \nabla \psi \\
& = & {1 \over 2} h^2 \Delta \psi.
\end{eqnarray*}

And thus putting this into \eqref{doom}, we have
$$
G_2(h) \psi =  -{1 \over 2}  |D|^{2} \left( {h^{2} } |D| \psi \right) 
+|D| \left\{
{1 \over 2} h^2 \Delta \psi + h  |D| \left( {h } |D| \psi \right) \right\}.
$$
Using the relationship $\Delta = - |D|^2$ gives:
$$
G_2(h) \psi =  {1 \over 2}  \Delta \left( {h^{2} } |D| \psi \right) 
+
{1 \over 2} |D| \left( h^2 \Delta \psi\right)  + |D| \left( h  |D| \left( {h } |D| \psi \right) \right).
$$

\subsection{The $\alpha$ derivatives of $\N$}
We expand the solution $(h,\psi)$ in powers of $\alpha$:
\be\label{expansion}
h(\alpha) = \al h_1 +{1 \over 2} \al^2 h_2 + O(\al^3) \ \text{ and } \ \psi(\al) = \al \psi_1 + {1 \over 2} \al^2 \psi_2 + O(\al^3). 
\ee

We need to compute the expansion of the nonlinearity $\N$ in terms $\al$. We want expressions for $\partial_\al^n \N$, $n = 1,2,3$.

\subsubsection{For $\N_1$}
We have
$$
\N_1 := G(h) \psi - G_0 \psi.
$$
Plugging in \eqref{expansion}, we get:
\begin{eqnarray}
\N_1 & = & G(h(\al))\psi(\al) - G_0 \psi(\alpha) \notag \\ &= & G_1(\al h_1 + \al^2 h_2/2 + O(\al^3)) (\al \psi_1 + \al^2 \psi_2/2 + O(\al^3)) \label{ugly} \\
&& +G_2(\al h_1 + \al^2 h_2/2 + O(\al^3)) (\al \psi_1 + \al^2 \psi_2/2 + O(\al^3) ) + \dots. \notag
\end{eqnarray}
Note that $G_n(\al h) = \al^n G_n(\al)$. Similarly for all $n$ there is an $n$-multilinear operator $\tilde{G}_n(h_1,\dots,h_n)$ such that
$$
G_n(h) = \tilde{G}_n(h,\dots,h).
$$
Using these two facts in \eqref{ugly} and then collecting powers of $\al$ gives:
$$
\N_1(\alpha) = \al^2 G_1(h_1) \psi_1 + \al^3 \left( {1 \over 2} G_1(h_2) \psi_1 + {1 \over 2} G_1(h_1) \psi_2 + G_2(h_1) \psi_1 \right) + O(\al^4).
$$
So we have 
$$
\partial_\al \N_1 = 0,
$$
$$
\partial_\al^2 \N_1 = -2\left[\operatorname{div} \left( h \nabla \psi \right) + |D| \left( {h } |D| \psi \right)\right]$$
and 
$$
\partial_\al^3 \N_1 = 3 G_1(h_2) \psi_1 + 3 G_1(h_1) \psi_2 +6 G_2(h_1) \psi_1.$$
Or rather, if we use the above computations for $G_1$ and $G_2$:
\begin{equation}
\begin{split}
\partial_\al^3 \N_1 = & -3 \left[  \operatorname{div} \left( h_1  \nabla \psi_2 \right) + |D| \left( {h_1 } |D| \psi_2 \right) \right]\\
 &-3 \left[  \operatorname{div} \left( h_2  \nabla \psi_1 \right) + |D| \left( {h_2 } |D| \psi_1 \right) \right]\\
& +3 \left[    \Delta \left( {h_1^{2} } |D| \psi_1 \right) 
+
 |D| \left( h_1^2 \Delta \psi_1 \right)  + 2|D| \left( h_1  |D| \left( {h_1 } |D| \psi_1 \right) \right)\right].
 \end{split}
\end{equation}

\subsubsection{For $\N_2$}
Now let us compute the expansion of $\N_2$. Indeed, from \eqref{ww}, we have
\begin{eqnarray*}
\N_2 & := & \tau\left[
\operatorname{div} \left( { \nabla h \over \sqrt{1 + |\nabla h|^2}} \right)-\Delta h\right] - {1 \over 2 } |\nabla \psi|^2 + {(G(h) \psi + \nabla h \cdot \nabla \psi)^2\over 2 (1 + |\nabla h|^2)} \\
& = & I + II + III.
\end{eqnarray*}
Taylor expansion and substitution of \eqref{expansion} gives:
$$
{1 \over  \sqrt{1 + |\nabla h|^2}} = 1 - {1 \over 2} |\nabla h|^2 + O(|\nabla h|^4) = 1 - \alpha^2{1 \over 2} |\nabla h_1|^2  + O(\alpha^4).
$$
Using this and \eqref{expansion} in $I$ gives:
\begin{multline}
I= \tau\left[
\operatorname{div} \left( { \nabla h \over \sqrt{1 + |\nabla h|^2}} \right)-\Delta h\right]  
\\= \tau\left[\operatorname{div} \left(  \nabla h  \left(1 - \alpha^2{1 \over 2} |\nabla h_1|^2\right) \right)-\Delta h\right] +O(\al^4) 
\\=\tau \operatorname{div} \left(  \nabla h  \left( - \alpha^2{1 \over 2} |\nabla h_1|^2\right) \right)\ +O(\al^4) 
\\= -\alpha^3 {\tau \over 2} \operatorname{div}\left( \nabla h_1 | \nabla h_1|^2 \right) + O(\al^4).\\
\end{multline}

The similar expansion for $III$ is pretty easy:
$$
II = - {1 \over 2 } |\nabla \psi|^2 = -\alpha^2{1 \over 2} |\nabla \psi_1|^2 - \alpha^3 {1 \over 2}\nabla \psi_1 \cdot \nabla \psi_2 + O(\al^4).
$$

To deal with $III$ notice that \eqref{expansion} shows that
$(G(h) \psi + \nabla h \cdot \nabla \psi)^2 = O(\al^2)$. Also, Taylor expansion gives:
$$
{1 \over (1 + |\nabla h|^2)} = 1  + O(|\nabla h|^2) = 1 + O(\al^2) 
$$
So we have:
$$
III=  {(G(h) \psi + \nabla h \cdot \nabla \psi)^2\over 2 (1 + |\nabla h|^2)} = {1 \over 2} (G(h) \psi + \nabla h \cdot \nabla \psi)^2 + O(\al^4).
$$
Then, since $G_2(h) \psi = O(\al^3)$, we have, with \eqref{expansion}
\begin{equation}
\begin{split}
III = &{1 \over 2} \left( G_0 \psi + G_1(h) \psi  + \nabla h \cdot \nabla \psi\right)^2 +O(\al^4)\\
      = &{1 \over 2} \left( \al G_0 \psi_1 +{1 \over 2} \al^2 G_0 \psi_2  + \al^2 G_1(h_1) \psi_1  + \al^2 \nabla h_1 \cdot \nabla \psi_1\right)^2 +O(\al^4)\\
     = & \al^2 {1 \over 2} (G_0 \psi_1)^2 + \al^3 \left( G_0 \psi_1\right) \left(   {1 \over 2} G_0 \psi_2 + G_1(h_1)  \psi_1 + \nabla h_1 \cdot \nabla \psi_1 \right)   +O(\al^4).
\end{split}
\end{equation}
If we insert the expressions for $G_1$ and $G_0$ we get
\begin{equation}
\begin{split}
III=&\al^2 {1 \over 2} (|D| \psi_1)^2 + \al^3(|D|\psi_1) \left( {1 \over 2} |D| \psi_2 - \operatorname{div}(h_1 \nabla \psi_1)-|D|(h_1|D|\psi_1) + \nabla h_1 \cdot \nabla \psi_1 
\right)\\
&+O(\al^4)\\
=&\al^2 {1 \over 2} (|D| \psi_1)^2 + \al^3(|D|\psi_1) \left( {1 \over 2} |D| \psi_2 - h_1 \Delta \psi_1-|D|(h_1|D|\psi_1) 
\right)+O(\al^4).
\end{split}
\end{equation}
Differentiation with respect to $\al$ and organizing terms gives:
$$
\partial_\al \N_2 = 0.
$$
Again we differentiate:
$$
\partial_\al^2 \N_2  =  (|D| \psi_1)^2- |\nabla \psi_1|^2,
$$
giving \eqref{2ndvarN}.

The third derivative is:
\begin{eqnarray*}
\partial_\al^3 \N_2 & = & -3{\tau } \operatorname{div} \left( \nabla h_1 | \nabla h_1|^2 \right) -3 \nabla \psi_1 \cdot \nabla \psi_2 \\
&& + 6(|D|\psi_1) \left( {1 \over 2} |D| \psi_2 - h_1 \Delta \psi_1-|D|(h_1|D|\psi_1) 
\right).
\end{eqnarray*}
A quick reorganization of terms gives:
\begin{equation}
\begin{split}
\partial_\al^3 \N_2 = &-3{\tau } \operatorname{div} \left( \nabla h_1 | \nabla h_1|^2 \right)\\
 & \ \ - 6 |D| \psi_1 \left( h_1 \Delta \psi_1 + |D|(h_1|D|\psi_1) \right)\\
  & \ \ \ +3\left( [|D|\psi_1][|D|\psi_2]- \nabla \psi_1 \cdot \nabla \psi_2\right),
\end{split}
\end{equation}
which completes \eqref{3rdvarN}.

For $1d$  interfaces, the expansion of $\p_\alpha \mathcal{N}(p(\alpha))|_{\alpha = 0}$ produces six terms instead of the 
twelve terms in the $2d$ case.   


\section{The Quadratic Threshold for Regularity of the Flow Map}
\label{S:ill}

We begin by studying the second-order iteration of the solution map. 
 From \eqref{seconditerate} and \eqref{duhamelprop}, we have the following Duhamel terms
\begin{equation}
\mathcal{Q} = 2\int_0^t e^{(t - t') \mathcal{L}} \begin{pmatrix} 
 -\operatorname{div} \LC h_1 \nabla \psi_1 \RC - |D| \LC h_1 |D| \psi_1 \RC 
 \\\\  {1 \over 2} \LC |D| \psi_1 \RC ^2 - {1 \over 2} |\nabla \psi_1|^2
 \end{pmatrix} d t' .
\end{equation}
Let 
\begin{align*}
Q_1 (h, \psi) & \equiv  -\operatorname{div} \LC h \nabla \psi \RC - |D| \LC h |D| \psi \RC , \\
Q_2 (h, \psi) & \equiv  {1\over2} \LC | D | \psi \RC^2 - {1\over 2} \LV \nabla \psi \RV^2 .
\end{align*}
Then,
\begin{align*}
\widehat{Q_1}(t,\xi) & = \int_{\R^2} m_1(\xi,\eta)  \widehat{h}(t, \xi-\eta) \widehat{\psi}(t, \eta) d \eta, \\
\widehat{Q_2}(t,\xi) & =  \int_{\R^2} m_2(\xi,\eta) \widehat{\psi}(t, \xi-\eta) \widehat{\psi}(t, \eta) d \eta,
\end{align*}
where 
\begin{align*}
m_1(\xi,\eta) & \equiv  \xi \cdot \eta  - |\xi| |\eta|  =   |\xi| |\eta| \LC \cos\theta_1 - 1 \RC, \\
m_2 (\xi,\eta) & \equiv{1\over2}  \LC (\xi - \eta) \cdot \eta  + |\xi - \eta| |\eta| \RC
 =  {1\over 2} |\xi - \eta | |\eta| \LC \cos\theta_2 + 1 \RC
\end{align*}
and 
\begin{equation}\label{angledefinitions}
\begin{split}
\theta_1 &= \hbox{ angle between } \xi \hbox { and } \eta, \\
\theta_2 & = \hbox{ angle between } \xi  - \eta \hbox { and } \eta .
\end{split}
\end{equation}

Expanding $\mathcal{Q}$ in terms of the linear propagators yields
\begin{equation}\label{defnQ}
\int_0^t e^{(t-t')\mathcal{L}} 
 {\begin{pmatrix} Q_1\LC e^{t'\mathcal{L}} \begin{pmatrix} h_0 \\ \psi_0 \end{pmatrix}\RC \\   Q_2\LC e^{t'\mathcal{L}} \begin{pmatrix} h_0 \\ \psi_0 \end{pmatrix}\RC \end{pmatrix}} dt'
= \int_0^t  \begin{pmatrix} \mathcal{Q}_1 \\ \mathcal{Q}_2 \end{pmatrix} dt'
\end{equation}
for the implicitly defined $\mathcal{Q}_j$.  
Taking the Fourier Transform of the above second order iterations yields the following set of terms:
\begin{align*}
\widehat{\mathcal{Q}_1} 
& =  \widehat{L}_1 (\xi,t-t') \widehat{Q}_1(\xi, t') + \widehat{L}_2(\xi, t-t') \widehat{Q}_2(\xi,t')        \\
& = \int \widehat{L}_1(\xi,t-t') m_1(\xi,\eta) \LC \widehat{L}_1(\xi-\eta, t') \widehat{h}(\xi-\eta) + \widehat{L}_2(\xi-\eta,t') \widehat{\psi}(\xi-\eta) \RC \\
& \quad \quad \times 
\LC \widehat{L}_3(\eta,t') \widehat{h}(\eta) + \widehat{L}_1(\eta,t') \widehat{\psi}(\eta) \RC  d \eta \\
& \quad +  \int \widehat{L}_2(\xi,t-t') m_2(\xi,\eta) \LC \widehat{L}_3(\xi-\eta, t') \widehat{h}(\xi-\eta) + \widehat{L}_1(\xi-\eta,t') \widehat{\psi}(\xi-\eta) \RC \\
& \quad \quad \times 
\LC \widehat{L}_3(\eta,t') \widehat{h}(\eta) + \widehat{L}_1(\eta,t') \widehat{\psi}(\eta) \RC  d \eta 
\end{align*}
and
\begin{align*}
\widehat{\mathcal{Q}_2} 
& =  \widehat{L}_3 (\xi,t-t') \widehat{Q}_1(\xi, t') + \widehat{L}_1(\xi, t-t') \widehat{Q}_2(\xi,t')        \\
& = \int \widehat{L}_3(\xi,t-t') m_1(\xi,\eta) \LC \widehat{L}_1(\xi-\eta, t') \widehat{h}(\xi-\eta) + \widehat{L}_2(\xi-\eta,t') \widehat{\psi}(\xi-\eta) \RC \\
& \quad \quad \times 
\LC \widehat{L}_3(\eta,t') \widehat{h}(\eta) + \widehat{L}_1(\eta,t') \widehat{\psi}(\eta) \RC  d \eta \\
& \quad +  \int \widehat{L}_1(\xi,t-t') m_2(\xi,\eta) \LC \widehat{L}_3(\xi-\eta, t') \widehat{h}(\xi-\eta) + \widehat{L}_1(\xi-\eta,t') \widehat{\psi}(\xi-\eta) \RC \\
& \quad \quad \times 
\LC \widehat{L}_3(\eta,t') \widehat{h}(\eta) + \widehat{L}_1(\eta,t') \widehat{\psi}(\eta) \RC  d \eta .
\end{align*}

Of the 16 terms in the two Duhamel operators $\mathcal{Q}_1$ and $\mathcal{Q}_2$, we look for a dominant lower bound term that violates an assumed upper bound.   This lower bound will arise from the following construction:

\subsection{Candidate for Breakdown of Flow Map Regularity}

For $N$ large and $\delta$ small we set
\begin{equation}\label{illposeddata}  \begin{split} 
\widehat{h} & = 0 , \\
\widehat{\psi} &= N^{-(s+1)} {1}_{N  \leq |\xi | \leq 2N} {1}_{| \theta | \leq \delta },
\end{split} \end{equation}
where $\xi = |\xi| e^{i\theta}$; hence, the support of $\widehat{\psi}$ lies on a small arc far from the origin.  Then our expression for $\mathcal{Q}_1$ and $\mathcal{Q}_2$ reduces to  the following four terms:
\begin{equation}
\begin{split}
 A &= \int \widehat{L}_1(\xi,t-t') m_1(\xi,\eta)  \widehat{L}_2(\xi-\eta, t') \widehat{L}_1 (\eta,t') \widehat{\psi}(\xi-\eta) \widehat{\psi}(\eta) d \eta, \\
B &= \int \widehat{L}_2(\xi,t-t') m_2(\xi,\eta)  \widehat{L}_1(\xi-\eta, t') \widehat{L}_1 (\eta,t') \widehat{\psi}(\xi-\eta) \widehat{\psi}(\eta) d \eta, \\
C &= \int \widehat{L}_3(\xi,t-t') m_1(\xi,\eta)  \widehat{L}_2(\xi-\eta, t') \widehat{L}_1 (\eta,t') \widehat{\psi}(\xi-\eta) \widehat{\psi}(\eta) d \eta, \\
D &= \int \widehat{L}_1(\xi,t-t') m_2(\xi,\eta)  \widehat{L}_1(\xi-\eta, t') \widehat{L}_1 (\eta,t') \widehat{\psi}(\xi-\eta) \widehat{\psi}(\eta) d \eta. \end{split}
\end{equation}

We will show that $D$ dominates the other terms and gives the required lower bound.  Since $m_1$ and $m_2$ have explicit dependence on the angle between $\xi$ and $\eta$, we first prove an auxiliary lemma that explains how the support of our 
constructed $\psi$ affects the support where $\psi(\xi-\eta) \psi(\eta)$ is nontrivial.   

In the following we write $A\lesssim B$ if there exists a universal constant $C$ such that $A \leq C B$.  
Likewise we write $A \gtrsim B$ if there exists a $C$ with $A \geq CB$.  Finally,  write $A \sim B$ if $A\lesssim B $ and $A \gtrsim B$.

\begin{lem}  \label{auxiliarylemmasupport}
If $N$ is large and $\delta$ is small then  
$\hat{\psi}(\xi-\eta)\hat{\psi}(\eta)$
has support in a set with 
\begin{equation} \label{xisame}
|\xi| \sim N \hbox { and } |\theta| \leq  \delta
\end{equation}
where $\xi  = |\xi|e^{i\theta}$ .
Furthermore, we have
\begin{align}
\LV m_1(\xi,\eta) \RV & \lesssim  N^2\delta^2 \label{m1same},  \\
\LV m_2(\xi,\eta) \RV & \sim  N^2 \label{m2same} .
\end{align}
\end{lem}
\begin{proof}

Supose $\xi -\eta$ and $\eta$ both lie in the same sector.  Then by the  parallelogram law, $\xi$ lies in that sector too, and $|\theta| \leq \delta $.    Thus, we also have $|\theta_1| \leq \delta$ and $| \theta_2 | \leq 2 \delta$.  Therefore,
\begin{align*}
&|\cos\theta_1 - 1 |  \leq \delta^2/2, \\
&1 \leq |\cos\theta_2 + 1 |  \leq 2
\end{align*}
for $\delta \leq {\pi \over2}$.

Obviously we have $|\xi|\leq 4N$. On the other hand, by the law of cosines $|\xi|^2 = |\xi - \eta |^2 + |\eta|^2 - 2 |\xi - \eta| |\eta| \cos(\pi - \theta_2)$,
so
\begin{align*}
|\xi|^2 &\geq 2N^2 - 2N^2\cos(\pi - 2\delta) = 2N^2 (1 + \cos 2\delta) \geq 2N^2.
\end{align*}
Hence $|\xi| \sim N$ and therefore,  
\[
 |m_1(\xi,\eta)|  = |\xi| |\eta| \LV \cos\theta_1 -  1 \RV \lesssim N^2\delta^2  
\]
and
\[
 |m_2 (\xi,\eta)|  = {1\over2} |\xi - \eta | |\eta| |\cos\theta_2 + 1 | \sim N^2 .
\]
\end{proof}

In the gravity-capillary case we have $\lambda_{gc} (r) = \sqrt{r^3 + r}$ so 
when $r \gg 1$, 
\begin{equation} \label{gcratio}
{\lambda_{gc}(r) \over |r| } \approx r^{1\over2}.
\end{equation}
Next, in order for $\cos(\lambda_{gc}(r) t') \geq {1\over 2}$ for $r \gg 1$ and all $0 \leq t' \leq t$, we require that 
$t \ll r^{-{3\over2}}$; 
therefore, we choose 
\begin{equation} \label{gctimebound}
0 \leq t \leq  {1\over 100}{ 1\over N^a} \ \hbox{ with } \  a \geq {3 \over 2}. 
\end{equation}
The following proposition follows from Lemma~\ref{auxiliarylemmasupport} and \eqref{gcratio}--\eqref{gctimebound}.

\begin{prop} \label{lowerboundprop}
On the support of $\widehat{\psi}(\xi-\eta)\widehat{\psi}(\eta)$ and for $0 \leq t' \leq t \leq  {1\over 100}{ 1\over N^a}$ with $a\geq3/2$ we have the following bounds:
\begin{eqnarray*}
\LV \!\!\!\!\! \right. && \!\!\!\!\! \widehat{L}_3(\xi, t-t') m_1 \widehat{L}_2(\xi - \eta,t') \widehat{L}_1(\eta,t') + \widehat{L}_1(\xi,t-t') m_2 \widehat{L}_1(\xi - \eta,t') \widehat{L}_1(\eta,t')  \\
&& \left. + \widehat{L}_2(\xi,t-t') m_2(\xi,\eta)  \widehat{L}_1(\xi-\eta, t') \widehat{L}_1 (\eta,t') + \widehat{L}_1(\xi,t-t') m_1(\xi,\eta)  \widehat{L}_2(\xi-\eta, t') \widehat{L}_1 (\eta,t') \RV \\
&& \ \ \ \ \sim  N^2.  
\end{eqnarray*}
\end{prop}

\begin{proof}
From \eqref{gctimebound} and Lemma~\ref{auxiliarylemmasupport} we know that for $N$ large
$$\lambda(\xi)(t-t'), \lambda(\xi)t', \lambda(\xi-\eta)t', \lambda(\eta)t' \leq {3\over100}N^{{3\over2}-a}\leq {3\over100}.$$
Therefore 
\[
1 \geq \widehat{L}_1(\xi,t-t'), \widehat{L}_1(\xi - \eta,t'), \widehat{L}_1(\eta,t')\geq {1\over2}
\]
and hence
\[
\LV \widehat{L}_1(\xi,t-t') m_2 \widehat{L}_1(\xi - \eta,t') \widehat{L}_1(\eta,t') \RV \geq {1\over 16} |\xi - \eta| |\eta|.
\]

On the other hand,
\begin{equation}\label{L3m1L2L1:cap}
\begin{split}
\LV \widehat{L}_3 m_1 \widehat{L}_2 \widehat{L}_1 \RV &\leq \LV \sin\LC \lambda(\xi)(t-t') \RC \RV {\lambda(\xi)  \over |\xi|}  |\xi| |\eta| \LC {\delta^2\over2} \RC \LV {\sin \LC\lambda(\xi-\eta)t'\RC \over \lambda(\xi-\eta)} \RV |\xi -\eta|\\
& = {\delta^2\over2}|\xi-\eta||\eta| \lambda^2(\xi)(t-t')t' \LV {\sin \LC\lambda(\xi)(t-t')\RC \over \lambda(\xi)(t-t')} \RV \LV {\sin \LC\lambda(\xi-\eta)t'\RC \over \lambda(\xi-\eta)t'} \RV\\
&\leq {\delta^2\over2}|\xi-\eta||\eta| \lambda^2(\xi)(t-t')t' \leq {\delta^2\over2}\LC {3\over100} \RC^2 |\xi-\eta||\eta|.
\end{split}
\end{equation}
Similarly,
\begin{equation*}
\begin{split}
\LV \widehat{L}_2 m_2 \widehat{L}_1 \widehat{L}_1 \RV & \leq \LV \sin\LC \lambda(\xi)(t-t') \RC \RV {|\xi| \over\lambda(\xi)}  |\xi-\eta| |\eta|\\
& = |\xi-\eta||\eta| |\xi|(t-t') \LV {\sin \LC\lambda(\xi)(t-t')\RC \over \lambda(\xi)(t-t')} \RV \\
&\leq|\xi-\eta||\eta| |\xi|(t-t') \lesssim N^{3-a},
\end{split}
\end{equation*}
\begin{equation*}
\begin{split}
\LV \widehat{L}_1 m_1 \widehat{L}_2 \widehat{L}_1 \RV &\leq {\delta^2\over2}|\xi| |\eta|  \LV {\sin \LC\lambda(\xi-\eta)t'\RC \over \lambda(\xi-\eta)} \RV |\xi -\eta|\\
& = {\delta^2\over2}|\xi||\eta||\xi-\eta|t' \LV {\sin \LC\lambda(\xi-\eta)t'\RC \over \lambda(\xi-\eta)t'} \RV\\
&\leq  {\delta^2\over2}|\xi||\eta||\xi-\eta|t' \lesssim \delta^2N^{3-a},
\end{split}
\end{equation*}
and we obtain the claimed lower bound.

\end{proof}

\subsection{Breakdown of Flow Map Regularity}

We now give a result relating to a $C^2$ regularity threshold.   Assume $(h,\psi)$ are
wellposed in $X^s \equiv H^{s+{1\over 2}} \otimes H^s$ for a range of $s\in \R$.  Then 
\begin{align*}
\sup_{0\leq t \leq T}\LN \mathcal{Q} \RN_{X^s} \leq  \LN \begin{pmatrix} h_0 \\ \psi_0 \end{pmatrix} \RN_{X^s}^2 .
\end{align*} 
Taking $(h_0, \psi_0)$ as in \eqref{illposeddata}, then
\begin{equation}   \label{supposedupperbound}
\LN \begin{pmatrix} h_0 \\ \psi_0 \end{pmatrix} \RN_{X^s} = \LN \psi_0 \RN_{H^s} = \LN \LA \xi \RA^s \widehat{\psi_0} \RN_{L^2} \sim \delta^{1/2}
\end{equation}
since the support lies on a set of area $\LC N^2 \delta \RC$.  So
\begin{equation}\label{assumeduppergravitycapillary}
\sup_{0 \leq t \leq 1} \LN \mathcal{Q} \RN_{X^s} \lesssim  \delta .
\end{equation}
Recall that $0 \leq t \leq {1\over 100} {1 \over N^a}$ for $a \geq {3 \over 2}$.  Note that from Lemma \ref{auxiliarylemmasupport} the region where $\mathcal{Q}$ is nontrivial lies in the high frequency sector of $\xi$. In particular, consider the high frequency sector  $E = \left\{ 2N \leq \LV \xi \RV \leq 4N, |\theta| \leq \delta/2 \right\}$, then we know $|E| \sim N^2\delta$.
From Proposition~\ref{lowerboundprop} we have 
\begin{align*}
\sup_{0 \leq t \leq T} \LN \mathcal{Q} \RN_{X^s} 
& \geq 2\sup_{0 \leq t \leq T} \LN    \LA \xi \RA^{s + {1 \over 2}} \int^t_0 \widehat{\mathcal{Q}_1} dt' \RN_{L^2(E)} + 2 \sup_{0 \leq t \leq T} \LN    \LA \xi \RA^{s} \int^t_0 \widehat{\mathcal{Q}_2} dt' \RN_{L^2(E)} \\
& \geq 2 \sup_{0 \leq t \leq T} \LN    \LA \xi \RA^{s} \int^t_0 \widehat{\mathcal{Q}_2} dt' \RN_{L^2(E)} \\ 
& \gtrsim \LN N^s \int_0^{1 \over 100N^a} \int_I N^2 N^{-2s-2}\ d\eta \RN_{L^2(E)} \\
& \gtrsim N^{-s} N^{-a} N^2 \delta (N^2\delta)^{1/2} \\
& = N^{3-a -s} \delta^{3/2},
\end{align*}
where $I = \left\{ \eta = |\eta| e^{i\theta}:\  N - 10 \leq |\eta| \leq N +10, |\theta| \leq \delta \right\}$. Taking $\delta = N^{-2\epsilon}$ for some $\epsilon>0$ then 
\[
\sup_{0 \leq t \leq T} \LN \mathcal{Q} \RN_{X^s}  \gg N^{0+}\delta,
\]
which violates \eqref{assumeduppergravitycapillary} 
when  $0 < 3 - a  - \epsilon - s \leq 3 - {3/2} - \epsilon - s = 3/2  - 3\epsilon - s$
or 
\[
s < 3/2 .
\]

As mentioned before, this $C^2$ threshold corresponds to solutions with $h \in H^s$ for $s < 2$.  Since $h$ is not necessarily Lipschitz,
the Dirichlet-Neumann operator may not make sense.  We next iterate to the third level, where the surface tension operator appears.

\section{Cubic Terms}

We take our candidate for regularity breakdown and see how it behaves under
the next iteration.  For $N$ large and $\delta$ small we again set
\begin{equation}\label{illposeddata:cubic}  
\begin{split} 
\widehat{h} & = 0 , \\
\widehat{\psi} &= N^{-(s+1)} {1}_{ N \leq |\xi | \leq 2N } {1}_{| \theta | \leq \delta } ,
\end{split} 
\end{equation}
where $\xi = |\xi| e^{i\theta}$; hence, the support of
$\widehat{\psi}$ lies on a small arc far from the origin.
We set
\begin{eqnarray*}
I = \{ \eta = | \eta | e^{i \theta} : N \leq | \eta | \leq
2N, \ | \theta | < \delta \}.
\end{eqnarray*}

From \eqref{thirditerate}, \eqref{cubic1} and \eqref{cubic2}, the third iterate now can be decomposed into a quadratic interaction between the first and second iterations and also a cubic interaction of the first iteration. To be more precise, we have
\begin{align}\label{thirditeratedecomposition}
\mathcal{C} & = \int^t_0 e^{(t-t')\mathcal{L}} \p_\alpha^3 \mathcal{N}(0, t')dt' = 3 \int^t_0 e^{(t-t')\mathcal{L}} \begin{pmatrix}  \tilde{Q}_1\\ \tilde{Q}_2 \end{pmatrix}(t')dt' + \int^t_0 e^{(t-t')\mathcal{L}} \begin{pmatrix}  C_1\\ C_2 \end{pmatrix}(t')dt' \\
& \equiv 3 \tilde{\mathcal{Q}} + \tilde{\mathcal{C}}, \notag
\end{align}
where
\begin{equation} \label{3rdcomponents} \begin{split}
\tilde{Q}_1 = &  -\operatorname{div} \LC h_1 \nabla \psi_2 \RC - |D| \LC h_1 |D|
\psi_2 \RC -  \operatorname{div} \LC h_2 \nabla \psi_1 \RC - |D| \LC h_2 |D| \psi_1 \RC, \\
\tilde{Q}_2 = &\  |D|\psi_1 |D|\psi_2 - \nabla \psi_1 \cdot \nabla \psi_2 , \\
C_1 = & {3} \LB \Delta \LC h_1^2  |D| \psi_1 \RC + |D| \LC h_1^2 \Delta \psi_1 \RC 
+ 2 |D|( h_1 |D| (h_1 |D| \psi_1)) \RB, \\
C_2 = & \ -6 \LB h_1 \Delta \psi_1  + |D|(h_1 |D| \psi_1)   \RB  |D| \psi_1  \\
& - 3 \tau \LB (3 h_{1,xx} h_{1,x}^2 + 3h_{1,yy} h_{1,y}^2+ 4h_{1,y} h_{1,x} h_{1,xy} + h_{1,x}^2 h_{1,yy} + h_{1,y}^2 h_{1,xx})  \RB.
\end{split}\end{equation}

First, we prove bounds on the cubic interactions resulting from the
cubic terms in the nonlinear expansion.  Let $p_1$ correspond to the appropriate cubic pseudodifferential
multiplier leading to the cubic terms in $C_1$.
Let $p_{2,1}$, $p_{2,2}$ correspond to the appropriate cubic pseudodifferential
multipliers leading to the cubic terms $(h_1,h_1,\psi_1)$, respectively
$(h_1,h_1,h_1)$ in $C_2$. Expanding, we have
\begin{eqnarray*}
p_1 (\xi, \nu, \eta) & = &-3 |\xi| |\eta| ( |\xi| + |\eta| -2 |\xi -\nu| ), \\
p_{2,1} (\xi,\nu,\eta) & = &  6|\eta| \LB   | \nu|^2 - |\xi - \eta| |\nu| \RB  , \\
p_{2,2} (\xi,\nu,\eta) & = & -9 \tau \LB (\xi_1 -\eta_1 -\nu_1)\nu_1\eta_1^2 + (\xi_2 -\eta_2 -\nu_2)\nu_2\eta_2^2 \RB  \\
&& - \tau  \LB  (\xi_1 -\eta_1 -\nu_1)\nu_1\eta_2^2 + ( \xi_2 -\eta_2 -\nu_2)\nu_2\eta_1^2 \RB \\
 && - 12 \tau \LB (\xi_1 -\eta_1 - \nu_1)\nu_2\eta_1\eta_2 \RB .
\end{eqnarray*}

\begin{lem}  \label{auxiliarylemmasupport:1d}
If $N$ is large and $\delta$ is small then  
$\hat{\psi}(\xi - \eta - \nu) \hat{\psi}(\eta)\hat{\psi}(\nu)$
has support in a set with 
\begin{equation} \label{xisame:1d}
|\xi| \sim N \hbox { and } |\theta| \leq  \delta
\end{equation}
where $\xi  = |\xi|e^{i\theta}$ .
\end{lem}
\begin{proof}

From \eqref{illposeddata:cubic} we know that  $\xi -\eta - \nu$, $\eta$ and $\nu$ all lie in the same sector. By parallelogram law, $\eta+\nu$ also lies in the sector. Because $\xi = (\xi - \eta - \nu) + (\eta + \nu)$, then by parallelogram law, $\xi$ lies in that sector too, and $|\theta| \leq \delta $.  As for an estimate of $|\xi|$, obviously we have $|\xi| \leq |\xi - \eta - \nu| + |\eta| + |\nu| \leq 6N$.  

By law of cosines, we can find angles $\alpha_1$ and $\alpha_2$ between $\xi-\eta-\nu$ and $\eta+\nu$, $\eta$ and $\nu$ respectively such that
\begin{align*}
|\xi|^2 & = |\xi - \eta -\nu |^2 + |\eta + \nu|^2 - 2 |\xi - \eta - \nu| |\eta + \nu| \cos\alpha_1, \\
|\eta + \nu|^2 & = |\eta|^2 + |\nu|^2 - 2 |\eta| |\nu| \cos \alpha_2,
\end{align*}
so
\begin{align*}
|\eta + \nu|^2  \geq N^2 + N^2 - 2 N^2 \cos(\pi - 2\delta) = 2 N^2 \LC 1 + \cos(2 \delta ) \RC .
\end{align*}
Therefore,  
\[
 |\xi|^2 \geq N^2 +  2 N^2 \LC 1 + \cos(2 \delta ) \RC + 2 N^2 \sqrt{2\LC 1 + \cos(2 \delta ) \RC} \cos (2\delta) \geq 4N^2
\]
and thus
\[
 2N \leq |\xi| \leq 6 N.
\]

\end{proof}

The following lemma provides a useful bound on the pseudodifferential operators arising in the cubic expansion. 

\begin{lem}  \label{auxiliarylemmasupport:cubic}
If $N$ is large and $\delta$ is small then we have
\begin{align}
\LV p_1(\xi,\eta) \RV, \  \LV \right. & \!\! \left.  p_{2,1} (\xi,\eta) \RV  \lesssim  N^3 , \label{p21same} \\
\LV p_{2,2} (\xi,\eta) \RV   & \sim N^4 \label{p22same} .
\end{align}
\end{lem}
\begin{proof}
The bound on $p_{2,1}$ is immediate.  Since $\xi - \nu = (\xi - \eta - \nu) + \eta$, we know that $\xi - \nu$ also lies in the sector $|\theta| \leq \delta$ and $|\xi - \nu| \sim N$. Then the estimate of $p_1$ can be easily obtained from Lemma \ref{auxiliarylemmasupport:1d}.  

To finish the proof, denote
\[
\xi - \eta - \nu = r_1 e^{i\beta_1}, \quad \eta = r_2e^{i\beta_2}, \quad \nu = r_3 e^{i\beta_3},
\]
then $|\beta_j|\leq \delta, \ N\leq r_j\leq2N, \ j = 1, 2, 3$. Hence
\begin{align*}
|p_{2,2}| &\sim 9r_1r_3r_2^2 \cos\beta_1\cos\beta_3\cos^2\beta_2 +  9r_1r_3r_2^2 \sin\beta_1\sin3\beta_3\sin^2\beta_2 + 3 r_1r_3r_2^2 \cos\beta_1\cos\beta_3\sin^2\beta_2 \\
& \quad + 3 r_1r_3 r_2^2 \sin\beta_1\sin\beta_3\cos^2 \beta_2 + 12 r_1r_3r_2^2 \cos\beta_1\cos\beta_3\cos\beta_2\sin\beta_2\\
& \sim r_1r_3r_2^2 \sim N^4.
\end{align*}
\end{proof}
Recall that $0 \leq t \leq {1\over 100} {1 \over N^a}$ for $a \geq {3
  \over 2}$.  Note that from Lemma \ref{auxiliarylemmasupport:1d}
the region where $\tilde{\mathcal{C}}$ is nontrivial lies in the high
frequency sector of $\xi$. In particular, consider the high frequency
sector  $E = \left\{ 2N \leq \LV \xi \RV \leq 6N, |\theta| \leq \delta \right\}$, then we know $|E| \sim N^2\delta$.

We are left with understanding the Duhamel term
\begin{equation}
\tilde{\mathcal{C}} =  \int_0^t e^{(t - t') \mathcal{L}} \begin{pmatrix} 
C_1 (t', h_1, \psi_1) 
 \\  C_2 (t', h_1, \psi_1)
 \end{pmatrix} d t' .
\end{equation}
Then,
\begin{align*}
\widehat{C_1}(t,\xi) & = \int_{\R^2} \int_{\R^2} p_1(\xi,\nu,\eta)
\widehat{h}_1(t, \xi-\eta-\nu) \widehat{h}_1(t, \nu) \widehat{\psi}_1(t,
\eta) d \eta d \nu \\
\widehat{C_2}(t,\xi) & =  \int_{\R^2} \int_{\R^2} \LB  p_{2,1} (\xi,\nu,\eta)
\widehat{h} (t, \xi-\eta-\nu) \widehat{\psi}_1(t, \nu) \widehat{\psi}_1(t,
\eta) \right.\\
& \quad \left. +  p_{2,2} (\xi,\nu,\eta)
\widehat{h}_1(t, \xi-\eta-\nu) \widehat{h}_1(t, \nu) \widehat{h}_1(t,
\eta) \RB d \eta d\nu .
\end{align*}

Expanding $\tilde{\mathcal{C}}$ in terms of the linear propagators yields
\[
\int_0^t e^{(t-t')\mathcal{L}} 
 {\begin{pmatrix} C_1\LC e^{t'\mathcal{L}} \begin{pmatrix} h_0 \\ \psi_0 \end{pmatrix}\RC \\   C_2\LC e^{t'\mathcal{L}} \begin{pmatrix} h_0 \\ \psi_0 \end{pmatrix}\RC \end{pmatrix}} dt'
= \int_0^t  \begin{pmatrix} \mathcal{C}_1 \\ \mathcal{C}_2 \end{pmatrix} dt'
\]
for the implicitly defined $\mathcal{C}_j$.  
Taking the Fourier Transform of the above second order iterations yield the following set of terms:
\begin{align*}
\widehat{\mathcal{C}_1} 
& =  \widehat{L}_1 (\xi,t-t') \widehat{C}_1(\xi, t') +
\widehat{L}_3(\xi, t-t') \widehat{C}_2(\xi,t')    ,    
\end{align*}
implying
\begin{align*}
\widehat{\mathcal{C}_1} 
& = \int \int \widehat{L}_1(\xi,t-t') p_1(\xi,\nu,\eta) \LC \widehat{L}_1(\nu,t') \widehat{h}(\nu) +
\widehat{L}_3( \nu,t')
\widehat{\psi}(\nu) \RC  \\
& \quad \quad \times 
\LC \widehat{L}_1(\xi-\eta-\nu, t') \widehat{h}(\xi-\eta-\nu) + \widehat{L}_3(\xi-\eta-\nu,t') \widehat{\psi}(\xi-\eta-\nu) \RC \\
& \quad \quad \times 
\LC \widehat{L}_2(\eta,t') \widehat{h}(\eta) + \widehat{L}_1(\eta,t')
\widehat{\psi}(\eta) \RC d \eta d \nu  \\
& \quad + \int \int \widehat{L}_3(\xi,t-t') p_{2,1} (\xi,\nu,\eta) \LC \widehat{L}_2(\nu,t') \widehat{h}(\nu) +
\widehat{L}_1( \nu,t')
\widehat{\psi}(\nu) \RC  \\
& \quad \quad \times 
\LC \widehat{L}_1(\xi-\eta-\nu, t') \widehat{h}(\xi-\eta-\nu) + \widehat{L}_3(\xi-\eta-\nu,t') \widehat{\psi}(\xi-\eta-\nu) \RC  \\
& \quad \quad \times 
\LC \widehat{L}_2(\eta,t') \widehat{h}(\eta) + \widehat{L}_1(\eta,t')
\widehat{\psi}(\eta) \RC d \eta d \nu  \\
& \quad  +  \int \int \widehat{L}_3(\xi,t-t') p_{2,2} (\xi,\nu,\eta) \LC \widehat{L}_1(\nu,t') \widehat{h}(\nu) +
\widehat{L}_3( \nu,t')
\widehat{\psi}(\nu) \RC  \\
& \quad \quad \times 
\LC \widehat{L}_1(\xi-\eta-\nu, t') \widehat{h}(\xi-\eta-\nu) + \widehat{L}_3(\xi-\eta-\nu,t') \widehat{\psi}(\xi-\eta-\nu) \RC \\
& \quad \quad \times 
\LC \widehat{L}_1(\eta,t') \widehat{h}(\eta) + \widehat{L}_3(\eta,t')
\widehat{\psi}(\eta) \RC d \eta d \nu  
\end{align*}
and a similar expression for $\mathcal{C}_2$  
\begin{align*}
\widehat{\mathcal{C}_2} 
& =  \widehat{L}_2 (\xi,t-t') \widehat{C}_1(\xi, t') +
\widehat{L}_1(\xi, t-t') \widehat{C}_2(\xi,t')        ,
\end{align*}
which implies
\begin{align*}
\widehat{\mathcal{C}_2} 
& = \int \int \widehat{L}_2(\xi,t-t') p_1(\xi,\nu,\eta) \LC \widehat{L}_1(\nu,t') \widehat{h}(\nu) +
\widehat{L}_3( \nu,t')
\widehat{\psi}(\nu) \RC  \\
& \quad \quad \times 
\LC \widehat{L}_1(\xi-\eta-\nu, t') \widehat{h}(\xi-\eta-\nu) + \widehat{L}_3(\xi-\eta-\nu,t') \widehat{\psi}(\xi-\eta-\nu) \RC   \\
& \quad \quad \times 
\LC \widehat{L}_2(\eta,t') \widehat{h}(\eta) + \widehat{L}_1(\eta,t')
\widehat{\psi}(\eta) \RC d \eta d \nu  \\
& \quad + \int \int \widehat{L}_1(\xi,t-t') p_{2,1} (\xi,\nu,\eta) \LC \widehat{L}_2(\nu,t') \widehat{h}(\nu) +
\widehat{L}_1( \nu,t')
\widehat{\psi}(\nu) \RC \\
& \quad \quad \times 
 \LC \widehat{L}_1(\xi-\eta-\nu, t') \widehat{h}(\xi-\eta-\nu) + \widehat{L}_3(\xi-\eta-\nu,t') \widehat{\psi}(\xi-\eta-\nu) \RC \\
& \quad \quad \times 
\LC \widehat{L}_2(\eta,t') \widehat{h}(\eta) + \widehat{L}_1(\eta,t')
\widehat{\psi}(\eta) \RC d \eta d \nu  \\
& \quad  +  \int \int \widehat{L}_1(\xi,t-t') p_{2,2} (\xi,\nu,\eta)  \LC \widehat{L}_1(\nu,t') \widehat{h}(\nu) +
\widehat{L}_3( \nu,t')
\widehat{\psi}(\nu) \RC \\
& \quad \quad \times 
\LC \widehat{L}_1(\xi-\eta-\nu, t') \widehat{h}(\xi-\eta-\nu) + \widehat{L}_3(\xi-\eta-\nu,t') \widehat{\psi}(\xi-\eta-\nu) \RC \\
& \quad \quad \times 
\LC \widehat{L}_1(\eta,t') \widehat{h}(\eta) + \widehat{L}_3(\eta,t')
\widehat{\psi}(\eta) \RC d \eta d \nu.
\end{align*}

As before, we may only look at $\mathcal{C}_2$. Setting $h_0 = 0$ and collecting
like terms, we obtain
\begin{align*}
\widehat{\mathcal{C}_2} 
& = \int \int \LC \widehat{L}_2(\xi,t-t') p_1(\xi,\nu,\eta) \widehat{L}_3(\xi-\eta-\nu,t')\widehat{L}_3(\nu,t')\widehat{L}_1(\eta,t') \right. \\
& \quad + \widehat{L}_1(\xi,t-t') p_{2,1} (\xi,\nu,\eta) \widehat{L}_3(\xi-\eta-\nu,t') \widehat{L}_1( \nu,t') \widehat{L}_1( \eta,t') \\
& \quad \left. + \widehat{L}_1(\xi,t-t') p_{2,2} (\xi,\nu,\eta) \widehat{L}_3(\xi-\eta-\nu,t')  \widehat{L}_3(\nu,t')  \widehat{L}_3(\eta,t') \RC \widehat{\psi}(\xi-\eta-\nu)\widehat{\psi}(\nu)\widehat{\psi}(\eta) d \eta d \nu.
\end{align*}
Choosing $0\leq t' \leq t \leq {1\over 100 N^{3/2}}$, we can bound
\begin{equation*}
|\widehat{L}_1(\xi,t-t')|, \quad |\widehat{L}_1( \nu,t')|, \quad |\widehat{L}_1( \eta,t')| \geq {1\over 2}.
\end{equation*}
Moreover, we have
\begin{align*}
\LV \widehat{L}_2 p_1 \widehat{L}_3 \widehat{L}_3 \widehat{L}_1 + \widehat{L}_1 p_{2,1} \widehat{L}_3 \widehat{L}_1 \widehat{L}_1 \RV = \LV \LB \widehat{L}_2(\xi)\widehat{L}_3(\nu) p_1 + \widehat{L}_1(\xi) \widehat{L}_1(\nu) p_{2,1} \RB \widehat{L}_3(\xi - \eta - \nu) \widehat{L}_1(\eta) \RV.
\end{align*}
Since
\begin{align*}
\widehat{L}_2(\xi)\widehat{L}_3(\nu) = \sin\LC \lambda(\xi)(t-t') \RC {|\xi|\over \lambda(\xi)} \cdot \sin\LC \lambda(\nu)t' \RC {\lambda(\nu) \over |\nu|} \ll 1,
\end{align*}
and by Lemma \ref{auxiliarylemmasupport:cubic},
\begin{eqnarray*}
\LV \widehat{L}_2 p_1 \widehat{L}_3 \widehat{L}_3 \widehat{L}_1 + \widehat{L}_1 p_{2,1} \widehat{L}_3 \widehat{L}_1 \widehat{L}_1 \RV  &\sim & N^3 \sin \LC \lambda(\xi -\eta - \nu)t' \RC {\lambda(\xi - \eta - \nu) \over |\xi -\eta - \nu|} \\
&\sim& N^{7/2} \sin\LC \lambda(\xi -\eta - \nu)t' \RC.
\end{eqnarray*}

The third multiplier in $\widehat{\mathcal{C}_2}$ can be estimated in the following way:
\begin{align*}
\LV \widehat{L}_1 \right. & \left. p_{2,2} \widehat{L}_3 \widehat{L}_3 \widehat{L}_3 \RV  \sim N^4\sin\LC \lambda(\xi - \eta -\nu)t' \RC\sin\LC \lambda(\nu)t' \RC\sin\LC \lambda(\eta)t' \RC \cdot {\lambda(\xi - \eta - \nu)\lambda(\nu)\lambda(\eta) \over |\xi - \eta - \nu| |\nu| |\eta|} \\
& \sim N^{11/2} \sin\LC \lambda(\xi - \eta -\nu)t' \RC\sin\LC \lambda(\nu)t' \RC\sin\LC \lambda(\eta)t' \RC \\
& = {N^{11/2} \over 4} \LB \sin\LC \lambda(\xi - \eta - \nu) - \lambda(\nu) + \lambda(\eta) \RC t'   - \sin\LC \lambda(\xi - \eta - \nu) + \lambda(\nu) + \lambda(\eta) \RC t'   \right.   \\
& \quad \quad \quad  \left. + \sin\LC \lambda(\xi - \eta - \nu) + \lambda(\nu) - \lambda(\eta)  \RC t' 
-  \sin\LC \lambda(\xi - \eta - \nu) - \lambda(\nu) - \lambda(\eta) \RC t' \RB .
\end{align*}

Therefore,
\begin{align*}
\sup_{0 \leq t \leq T} \LN \tilde{\mathcal{C}} \RN_{X^s} 
& \geq 6\sup_{0 \leq t \leq T} \LN    \LA \xi \RA^{s + {1 \over 2}} \int^t_0 \widehat{\mathcal{C}_1} dt' \RN_{L^2(E)} + 6 \sup_{0 \leq t \leq T} \LN    \LA \xi \RA^{s} \int^t_0 \widehat{\mathcal{C}_2} dt' \RN_{L^2(E)} \\
& \geq 6 \sup_{0 \leq t \leq T} \LN    \LA \xi \RA^{s} \int^t_0 \widehat{\mathcal{C}_2} dt' \RN_{L^2(E)} \\
& \gtrsim \sup_{0 \leq t \leq T} \LN    \LA \xi \RA^{s} \int^t_0 \int \int \LC\LV \widehat{L}_1 p_{2,2} \widehat{L}_3 \widehat{L}_3 \widehat{L}_3 \RV - \LV \widehat{L}_2 p_1 \widehat{L}_3 \widehat{L}_3 \widehat{L}_1 + \widehat{L}_1 p_{2,1} \widehat{L}_3 \widehat{L}_1 \widehat{L}_1 \RV\RC \right. \\
& \quad \quad \quad \quad \quad \quad \times  \left. \LV \widehat{\psi}\widehat{\psi}\widehat{\psi} \RV d\eta d\nu dt' \RN_{L^2(E)} .
\end{align*}

First, we have
\begin{align*}
\sup_{0 \leq t \leq T} & \LN    \LA \xi \RA^{s} \int^t_0 \int \int  \LV \widehat{L}_2 p_1 \widehat{L}_3 \widehat{L}_3 \widehat{L}_1 + \widehat{L}_1 p_{2,1} \widehat{L}_3 \widehat{L}_1 \widehat{L}_1 \RV \LV \widehat{\psi}\widehat{\psi}\widehat{\psi} \RV d\eta d\nu dt' \RN_{L^2(E)} \\
& \sim \LN N^s \int_0^{{1\over 100N^{3/2}}}\int_I\int_I N^{7/2}N^{-3s-3} \sin\LC \lambda(\xi - \eta - \nu)t' \RC d\eta d\nu dt'  \RN_{L^2(E)} \\
& \sim N^{1/2 - 2s} \LN \int_I \int_I {1 - \cos\LC \lambda(\xi - \eta - \nu)N^{-3/2}/100 \RC \over \lambda(\xi - \eta - \nu)} d\eta d\nu \RN_{L^2(E)} \\
& \sim N^{1/2 - 2s} N^{-3/2} |I|^2 |E|^{1/2} \sim N^{1/2 - 2s} N^{-3/2} N^4\delta^2 N \delta^{1/2} \\
& \sim N^{4 - 2s} \delta^{5/2}.
\end{align*}
Similarly,
\begin{align*}
\sup_{0 \leq t \leq T} & \LN    \LA \xi \RA^{s} \int^t_0 \int \int  \LV \widehat{L}_1 p_{2,2} \widehat{L}_3 \widehat{L}_3 \widehat{L}_3 \RV \LV \widehat{\psi}\widehat{\psi}\widehat{\psi} \RV d\eta d\nu dt' \RN_{L^2(E)} \\
& \sim N^{5/2 - 2s} \LN \int_I \int_I  {1 - \cos\LB\LC \lambda(\xi - \eta - \nu) - \lambda(\nu) + \lambda(\eta) \RC N^{-3/2}/100\RB \over  \lambda(\xi - \eta - \nu) - \lambda(\nu) + \lambda(\eta) } + \cdots \  d\eta d\nu \RN_{L^2(E)} \\
& \sim N^{5/2 - 2s} N^{-3/2} |I|^2 |E|^{1/2} \\
& \sim N^{6 - 2s} \delta^{5/2},
\end{align*}
where we have used the fact that in the support of $\widehat{\psi}(\xi - \eta - \nu) \widehat{\psi}(\eta)\widehat{\psi}(\nu)$,
\begin{align*}
& \LV  \lambda(\xi - \eta - \nu) - \lambda(\nu) + \lambda(\eta) \RV, \quad \LV \lambda(\xi - \eta - \nu) - \lambda(\nu) - \lambda(\eta) \RV, \\
& \LV \lambda(\xi - \eta - \nu) + \lambda(\nu) + \lambda(\eta) \RV,  \quad \LV \lambda(\xi - \eta - \nu) + \lambda(\nu) - \lambda(\eta) \RV \quad \sim N^{3/2}.
\end{align*}
Thus, 
\begin{align}\label{cclowerbound}
\sup_{0 \leq t \leq T} \LN \tilde{\mathcal{C}} \RN_{X^s} \gtrsim N^{6 - 2s} \delta^{5/2}.
\end{align}

Now, we seek to prove bounds on the quadratic terms in the expression
for $\p_\alpha^3 \mathcal{N}$.  In particular, we have
\begin{align*}
\widehat{\tilde{Q}}_1 & = \int_{\R^2} m_1 (\widehat{h}_1 \widehat{\psi}_2 + \widehat{h}_2 \widehat{\psi}_1)\ d\eta, \\
\widehat{\tilde{Q}}_2 & = \int_{\R^2} 2 m_2 \widehat{\psi}_1 \widehat{\psi}_2 \ d\eta,
\end{align*}
where $\tilde{Q}_1$ and $\tilde{Q}_2$ are given in \eqref{3rdcomponents}.

As a result, the contribution to the cubic iteration from the
quadratic terms is
\begin{eqnarray*}
\widehat{ \tilde{\mathcal{Q}}_1 }  & = & \widehat{L}_1
\widehat{\tilde{Q}}_1 + \widehat{L}_3 \widehat{\tilde{Q}}_2 , \\
\widehat{ \tilde{\mathcal{Q}}_2 }  & = & \widehat{L}_2
\widehat{\tilde{Q}}_1 + \widehat{L}_1 \widehat{\tilde{Q}}_2 ,
\end{eqnarray*}
where
\begin{equation*}
 \begin{pmatrix}  \tilde{\mathcal{Q}}_1\\ \tilde{\mathcal{Q}}_2 \end{pmatrix} = e^{(t-t')\mathcal{L}} \begin{pmatrix}  \tilde{Q}_1\\ \tilde{Q}_2 \end{pmatrix}.
\end{equation*}

Let us take $\widehat{ \tilde{\mathcal{Q}}_2 }$ and analyze its size in $H^s$:
\begin{eqnarray*}
\widehat{ \tilde{\mathcal{Q}}_2 } & = & \int_0^t \int_0^s \int \widehat{L}_2 m_1 (
\widehat{L}_1 \widehat{h}_0 + \widehat{L}_3 \widehat{\psi}_0) \widehat{\mathcal{Q}}_2 (t') d \eta
dt'ds \\
&& + \int_0^t \int_0^s \int \widehat{L}_2 m_1 (
\widehat{L}_2 \widehat{h}_0 + \widehat{L}_1 \widehat{\psi}_0) \widehat{\mathcal{Q}}_1 (t') d \eta
dt'ds \\
&& + 2 \int_0^t \int^s_0 \int \widehat{L}_1 m_2 (
\widehat{L}_2 \hat{h}_0 + \widehat{L}_1 \widehat{\psi}_0) \widehat{ \mathcal{Q}}_2 (t') d \eta
dtds,
\end{eqnarray*}
where we have used \eqref{2ndvarN} and \eqref{h2psi2}, and $\mathcal{Q}_j$ is defined in \eqref{defnQ}.
Taking $0\leq t \leq {1\over 100N^a}$ for $a\geq 3/2$, bounding $\widehat{ \tilde{\mathcal{Q}}_2 }$ in $H^s$ gives 
\begin{align*}
\LN \widehat{ \tilde{\mathcal{Q}}_2 } \RN_{H^s(E)} & = \LN \LA \xi \RA^{s} \int_0^t \int_0^s \int \widehat{L}_2 m_1\LC \widehat{L}_3 \widehat{\psi}_0 \widehat{\mathcal{Q}}_2 + \widehat{L}_1 \widehat{\psi}_0 \widehat{\mathcal{Q}}_1 \RC + 2 \widehat{L}_1 m_2 \widehat{L}_1 \widehat{\psi}_0 \widehat{ \mathcal{Q}}_2 \ d \eta
dtds   \RN_{L^2(E)} \\
& \lesssim N^s N^{-2a} N^2 N^2 N^{-3s-3} |I|^2 |E|^{1/2},
\end{align*}
where we have used that
\begin{align*}
& \LV \widehat{L}_2 m_1 \widehat{L}_3 + \widehat{L}_2 m_1 \widehat{L}_1 + \widehat{L}_1 m_2 \widehat{L}_1 \RV \sim N^2 \\
& \LV \widehat{ {\mathcal{Q}}_j } \RV \lesssim \int_I N^2 N^{-2s} \ d\eta \sim N^2N^{-2s} |I|.
\end{align*}
Therefore choosing $a = 3/2$ we see that
\begin{equation*}
\LN \widehat{ \tilde{\mathcal{Q}}_2 } \RN_{H^s(E)} \lesssim N^{6 - 2a -2s} \delta^{5/2} = N^{3 - 2s} \delta^{5/2}.
\end{equation*}
A similar result holds for $\LN \widehat{ \tilde{\mathcal{Q}}_1 } \RN_{H^{s+{1\over2}}(E)}$. Hence it is of lower order to the largest
cubic term.

For $2d$ interface problem we can now finish the 
\begin{proof}[Proof of Theorem \ref{illposednesstheorem}]
We know that
\begin{equation}   \label{supposedupperbound:cubic}
\LN \begin{pmatrix} h_0 \\ \psi_0 \end{pmatrix} \RN_{X^s} = \LN \psi_0
\RN_{H^s} = \LN \LA \xi \RA^s \widehat{\psi}_0 \RN_{L^2} \sim \delta^{\frac12}
\end{equation}
since the support lies on a set of area $\LC N^2 \delta \RC$.  So
\begin{equation}\label{assumeduppergravitycapillary:cubic}
\sup_{0 \leq t \leq 1} \LN \mathcal{C} \RN_{X^s} \lesssim \delta^{3/2}.
\end{equation}
Taking $\delta = N^{-2\epsilon}$, from \eqref{cclowerbound} we have
\begin{align*}
\sup_{0 \leq t \leq T} \LN {\mathcal{C}} \RN_{X^s} \gtrsim N^{6 - 2s} \delta^{5/2} \gg \delta^{3/2},
\end{align*}
which violates \eqref{assumeduppergravitycapillary:cubic} when 
\begin{equation*} 
6 - 2s - 2\epsilon > 0
\end{equation*}
or
\begin{equation*}
s < 3 - {\epsilon}.
\end{equation*}

\end{proof}

\section{$1d$ interface}
\label{app:1d}
We make a few comments on the proof of the regularity threshold in $1d$. 
As in the $2d$ calculation, we restrict $\psi_0$ to an interval at high frequency and set 
\begin{equation}\label{illposeddata:cubic1d}  \begin{split} 
\widehat{h} & = 0 \\
\widehat{\psi} &= N^{-(s+\frac12)} {1}_{ N \leq \xi  \leq 2N }. 
\end{split} \end{equation}
As in the $2d$ case, the quadratic iteration provides a regularity threshold in $X^s = H^{s+ \frac12} \otimes H^s$ for $s < 3/2$, which 
is borderline Lipschitz for $h$ in one dimension as required to make sense of the expansion in \cite{NR1}.  Iterating further we find, as in the $2d$ calculation that 
\begin{align}\label{thirditeratedecomposition1d}
\mathcal{C} & = \int^t_0 e^{(t-t')\mathcal{L}} \p_\alpha^3 \mathcal{N}(0, t')dt' = 6 \int^t_0 e^{(t-t')\mathcal{L}} \begin{pmatrix}  \tilde{Q}_1\\ \tilde{Q}_2 \end{pmatrix}(t')dt' + \int^t_0 e^{(t-t')\mathcal{L}} \begin{pmatrix}  C_1\\ C_2 \end{pmatrix}(t')dt' \\
& \equiv 6\tilde{\mathcal{Q}} + \tilde{\mathcal{C}}, \notag
\end{align}
where
\begin{equation} \label{3rdcomponents1d} \begin{split}
\tilde{Q}_1 = & - \p_x\LC h_1 \p_x \psi_2 \RC - |\p_x| \LC h_1 |\p_x|
\psi_2 \RC -  \p_x \LC h_2 \p_x \psi_1 \RC - |\p_x| \LC h_2|\p_x| \psi_1 \RC, \\
\tilde{Q}_2 = &\ |\p_x|\psi_1 |\p_x| \psi_2 - \p_x \psi_1  \p_x \psi_2, \\
C_1 = & {3} \LB \p_{xx} \LC h_1^2  |\p_x| \psi_1 \RC + |\p_x| \LC h_1^2 \p_{xx} \psi_1 \RC 
+ 2 |\p_x|( h_1|\p_x| (h_1|\p_x| \psi_1)) \RB, \\
C_2 = & \ -6 \LB   h_1 \p_{xx} \psi_1 + |\p_x|(h_1 |\p_x| \psi_1)   \RB  |\p_x| \psi_1 
- 9 \tau \LB  \p_{xx} h_1 (\p_x h_1)^2  \RB  . 
\end{split}\end{equation}

As in the $2d$ case,  the terms arising from the second-order iteration, $\tilde{Q}_j$, are lower order, and we can restrict our attention 
to the terms arising as cubic interactions of the linear propagators $h_1, \psi_1$.  The cubic order pseudodifferential operator multipliers are simpler in the $1d$ case, namely we have
\begin{align*}
p_1 & = -3 |\xi| |\eta| \LC |\xi| + |\eta| - 2|\xi - \nu| \RC, \\
p_{2,1} & = 6 |\eta| \LB |\nu|^2 - |\xi - \eta| |\nu| \RB , \\
p_{2,2}&  = - 9 \tau (\xi - \eta -\nu) \nu \eta^2, 
\end{align*}
and  $\widehat\psi(\xi - \eta - \nu) \widehat\psi(\eta) \widehat\psi(\nu)$ has support on the set $E = \{3N \leq \xi \leq 6N\}$.  
As in Lemma~\ref{auxiliarylemmasupport:cubic} 
\begin{align*}
p_1 & \lesssim N^3, \\
p_{2,1} & \lesssim N^3 ,\\
p_{2,2} & \sim N^4.
\end{align*}
We again concentrate on $\mathcal{C}_2$ which produces the largest term of the cubic expansion. Again we have
\begin{align*}
\widehat{\mathcal{C}_2} 
& = \int \int \LC \widehat{L}_2(\xi,t-t') p_1(\xi,\nu,\eta) \widehat{L}_3(\xi-\eta-\nu,t')\widehat{L}_3(\nu,t')\widehat{L}_1(\eta,t') \right. \\
& \quad + \widehat{L}_1(\xi,t-t') p_{2,1} (\xi,\nu,\eta) \widehat{L}_3(\xi-\eta-\nu,t') \widehat{L}_1( \nu,t') \widehat{L}_1( \eta,t') \\
& \quad \left. + \widehat{L}_1(\xi,t-t') p_{2,2} (\xi,\nu,\eta) \widehat{L}_3(\xi-\eta-\nu,t')  \widehat{L}_3(\nu,t')  \widehat{L}_3(\eta,t') \RC \widehat{\psi}(\xi-\eta-\nu)\widehat{\psi}(\nu)\widehat{\psi}(\eta) d \eta d \nu.
\end{align*}
We study the third multiplier in $\widehat{\mathcal{C}_2}$,
\begin{eqnarray*}
\LV \widehat{L}_1 p_{2,2} \widehat{L}_3 \widehat{L}_3 \widehat{L}_3 \RV,
\end{eqnarray*}
which is again bounded by
\begin{align*}
  {N^{11/2} \over 4} & \LB \sin\LC \lambda(\xi - \eta - \nu) - \lambda(\nu) + \lambda(\eta) \RC t' -  \sin\LC \lambda(\xi - \eta - \nu) - \lambda(\nu) - \lambda(\eta) \RC t' \right.\\
& \quad \quad \quad \left. - \sin\LC \lambda(\xi - \eta - \nu) + \lambda(\nu) + \lambda(\eta) \RC t' + \sin\LC \lambda(\xi - \eta - \nu) + \lambda(\nu) - \lambda(\eta) \RC t' \RB.
\end{align*}
Thus, if $I = \{ N \leq \xi \leq 2 N\}$, then
\begin{align*}
\sup_{0 \leq t \leq T} & \LN    \LA \xi \RA^{s} \int^t_0 \int \int  \LV \widehat{L}_1 p_{2,2} \widehat{L}_3 \widehat{L}_3 \widehat{L}_3 \RV \LV \widehat{\psi}\widehat{\psi}\widehat{\psi} \RV d\eta d\nu dt' \RN_{L^2(E)} \\
& \sim N^{4 - 2s} \LN \int_I \int_I  {1 - \cos\LB\LC \lambda(\xi - \eta - \nu) - \lambda(\nu) + \lambda(\eta) \RC N^{-3/2}/100\RB \over  \lambda(\xi - \eta - \nu) - \lambda(\nu) + \lambda(\eta) } + \cdots \  d\eta d\nu \RN_{L^2(E)} \\
& \sim N^{4 - 2s} N^{-3/2} |I|^2 |E|^{1/2} \\
& \sim N^{5 - 2s} ,
\end{align*}
where we have used the fact that in the support of $\widehat{\psi}(\xi - \eta - \nu) \widehat{\psi}(\eta)\widehat{\psi}(\nu)$,
\begin{align*}
& \LV  \lambda(\xi - \eta - \nu) - \lambda(\nu) + \lambda(\eta) \RV, \quad \LV \lambda(\xi - \eta - \nu) - \lambda(\nu) - \lambda(\eta) \RV, \\
& \LV \lambda(\xi - \eta - \nu) + \lambda(\nu) + \lambda(\eta) \RV,  \quad \LV \lambda(\xi - \eta - \nu) + \lambda(\nu) - \lambda(\eta) \RV \quad \sim N^{3/2}.
\end{align*}
As a result, 
\begin{align}\label{cclowerbound1d}
\sup_{0 \leq t \leq T} \LN \tilde{\mathcal{C}} \RN_{X^s} \gtrsim N^{5 - 2s} .
\end{align}

We know that
\begin{equation}   \label{supposedupperbound:cubic1d}
\LN \begin{pmatrix} h_0 \\ \psi_0 \end{pmatrix} \RN_{X^s} = \LN \psi_0
\RN_{H^s} = \LN \LA \xi \RA^s \widehat{\psi}_0 \RN_{L^2} \sim 1
\end{equation}
since the support lies on an interval of length $N$.  So
\begin{equation}\label{assumeduppergravitycapillary:cubic1d}
\sup_{0 \leq t \leq 1} \LN \mathcal{C} \RN_{X^s} \lesssim 1.
\end{equation}
Then from \eqref{cclowerbound1d} we have
\begin{align*}
\sup_{0 \leq t \leq T} \LN {\mathcal{C}} \RN_{X^s} \gtrsim N^{5 - 2s}  \gg 1,
\end{align*}
which violates \eqref{assumeduppergravitycapillary:cubic1d} when $5 - 2s  > 0$
or
\begin{equation*}
s < 5/2.
\end{equation*}

\section{Pure gravity problem}
\label{sec:puregravity}

As mentioned in the introduction we visit here the flow map regularity threshold for the pure gravity problem.  In this case, we gain nothing by going past the regularity threshold at the level of the quadratic terms, as the cubic interactions due to surface tension are responsible for the improved results above.  As a result, we proceed as in Section \ref{S:ill}.  In the pure gravity case, we recall from \eqref{lambdag} that we have $\lambda_{g} (r) = r^{1\over2}$.  Hence,
when $r \gg 1$, 
\begin{equation} \label{gratio}
{\lambda_{g}(r) \over |r| } \approx r^{-{1\over2}},
\end{equation}
and we choose 
\begin{equation} \label{gtimebound}
0 \leq t \leq  {1\over 100}{ 1\over N^a} \ \hbox{ with } \  a \geq {1 \over 2}. 
\end{equation}
in order for  $\cos(\lambda_{g}(r) t') \geq {1\over 2}$ for $r \gg 1$ and all $0 \leq t' \leq t$.   
Following as above we can show that  

\begin{prop} \label{lowerboundpropgravity}
Given $\widehat{L}_1,\widehat{L}_2,\widehat{L}_3$ defined using $\lambda_g$ in \eqref{lambdag}, on the support of $\widehat{\psi}(\xi-\eta)\widehat{\psi}(\eta)$ and for $0 \leq t' \leq t \leq {1\over 100}{ 1\over N^a}$ with $a\geq1/2$ we have the following bounds: 
\begin{eqnarray*}
\LV \widehat{L}_3(\xi, t-t') m_1 \widehat{L}_2(\xi - \eta,t') \widehat{L}_1(\eta,t') + \widehat{L}_1(\xi,t-t') m_2 \widehat{L}_1(\xi - \eta,t') \widehat{L}_1(\eta,t') \RV \sim N^2.  
\end{eqnarray*}
\end{prop}

\begin{proof}
Note, this statement is exactly analogous to that of Proposition \ref{lowerboundprop}.  The proof will follow quite similarly, except that we have now
$$\lambda(\xi)(t-t'), \lambda(\xi)t', \lambda(\xi-\eta)t', \lambda(\eta)t' \leq {1\over50}N^{{1\over2}-a}\leq {1\over50}.$$
Therefore,
\[
\widehat{L}_1(\xi,t-t'), \widehat{L}_1(\xi - \eta,t'), \widehat{L}_1(\eta,t')\geq {1\over2}
\]
and hence
\[
\LV \widehat{L}_1(\xi,t-t') m_2 \widehat{L}_1(\xi - \eta,t') \widehat{L}_1(\eta,t') \RV \geq {1\over 16} |\xi - \eta| |\eta|.
\]

On the other hand,
\begin{equation}\label{L3m1L2L1:grav}
\begin{split}
\LV \widehat{L}_3 m_1 \widehat{L}_2 \widehat{L}_1 \RV &\leq 2\LV \sin\LC \lambda(\xi)(t-t') \RC \RV {\lambda(\xi)  \over |\xi|}  |\xi| |\eta| \LC {\delta^2\over2} \RC \LV {\sin \LC\lambda(\xi-\eta)t'\RC \over \lambda(\xi-\eta)} \RV |\xi -\eta|\\
& = \delta^2|\xi-\eta||\eta| \lambda^2(\xi)(t-t')t' \LV {\sin \LC\lambda(\xi)(t-t')\RC \over \lambda(\xi)(t-t')} \RV \LV {\sin \LC\lambda(\xi-\eta)t'\RC \over \lambda(\xi-\eta)t'} \RV\\
&\leq \delta^2|\xi-\eta||\eta| \lambda^2(\xi)(t-t')t' \leq \delta^2 \LC {1\over50} \RC^2 |\xi-\eta||\eta|.
\end{split}
\end{equation}
So we get the claimed lower bound.

\end{proof}
\begin{rem}
Notice that in the above proposition we do not include the other two multipliers $\widehat{L}_2(\xi,t-t') m_2(\xi,\eta)  \widehat{L}_1(\xi-\eta, t') \widehat{L}_1 (\eta,t')$ and $\widehat{L}_1(\xi,t-t') m_1(\xi,\eta)  \widehat{L}_2(\xi-\eta, t') \widehat{L}_1 (\eta,t')$ like in Proposition \ref{lowerboundprop}. This is because in obtaining a lower bound of $\sup_{0\leq t\leq T} \|\mathcal{Q}\|_{Y^s}$, we only need to estimate $\| \sup_{0\leq t\leq T}\int^t_0 \mathcal{Q}_2 dt' \|_{H^s}$, which involves only two multipliers given in Proposition \ref{lowerboundpropgravity}. The estimates for the other two multipliers $\widehat{L}_2 m_2  \widehat{L}_1 \widehat{L}_1$ and $\widehat{L}_1 m_1  \widehat{L}_2 \widehat{L}_1$ are used at the third iterates. However in the gravity case, to obtain similar estimates for $\widehat{L}_2 m_2  \widehat{L}_1 \widehat{L}_1$ and $\widehat{L}_1 m_1  \widehat{L}_2 \widehat{L}_1$, one would need to require $a\geq 1$, which lowers the regularity threshold $s$. Therefore going up to the third iterates for the gravity waves does not help to improve the threshold exponent. 
\end{rem}

For the pure gravity case we follow the same argument with a few changes.
Assume $(h,\psi)$ are
wellposed in $Y^s \equiv H^{s-{1\over 2}} \otimes H^s$ for a range of $s\in \R$. 
Then 
\begin{align*}
\sup_{0\leq t \leq T}\LN \mathcal{Q} \RN_{Y^s} \leq  \LN \begin{pmatrix} h_0 \\ \psi_0 \end{pmatrix} \RN_{Y^s}^2.
\end{align*} 
Assuming initial data as in \eqref{illposeddata} and the resulting bound \eqref{supposedupperbound},  then
\[
\sup_{0 \leq t \leq 1} \LN \mathcal{Q} \RN_{Y^s} \lesssim \delta.
\]
Recall that $0 \leq t \leq {1\over 100} {1 \over N^a}$ for $a \geq {1/2}$.  We again consider  the high frequency sector  $E$ as before.   From Proposition~\ref{lowerboundpropgravity} we have again
\begin{align*}
\sup_{0 \leq t \leq T} \LN \mathcal{Q} \RN_{Y^s} 
& \geq 2\sup_{0 \leq t \leq T} \LN    \LA \xi \RA^{s - {1 \over 2}} \int^t_0 \widehat{\mathcal{Q}_1} dt' \RN_{L^2(E)} + 2\sup_{0 \leq t \leq T} \LN    \LA \xi \RA^{s} \int^t_0 \widehat{\mathcal{Q}_2} dt' \RN_{L^2(E)} \\
& \gtrsim N^{3-a -s} \delta^{3/2};
\end{align*}
therefore,  
\[
\sup_{0 \leq t \leq T} \LN \mathcal{Q} \RN_{Y^s}  \gg N^{0+}\delta
\]
 violates \eqref{supposedupperbound} 
so long as  
\[
0 < 3 - a  - \epsilon - s \leq 3 - 1/2 - \epsilon - s = 5/2  - \epsilon - s
\]
or 
\[
s < {5\over2} .
\]



\end{document}